\documentclass[12pt]{amsart}
\usepackage{amsmath, amsthm, amscd, amsfonts, amssymb, graphicx, color}
\usepackage[margin=2.8cm]{geometry}
\usepackage[bookmarksnumbered, colorlinks, plainpages]{hyperref}
\usepackage{tikz} 

\newtheorem{theorem}{Theorem}[section]
\newtheorem{lemma}[theorem]{Lemma}
\newtheorem{proposition}[theorem]{Proposition}
\newtheorem{corollary}[theorem]{Corollary}
\theoremstyle{definition}

\newtheorem{conjecture}[theorem]{Conjecture}

\newtheorem{problem}[theorem]{Problem}

\theoremstyle{remark}
\newtheorem{remark}[theorem]{Remark}
\numberwithin{equation}{section}

\allowdisplaybreaks
\begin{document}
\title{\textbf{On a class of sums with unexpectedly high
cancellation, and its applications}}
\author {\bf Ernie Croot, Hamed Mousavi}
\begin{center}
\address{Department of Mathematics, Georgia Tech, Atlanta, US.}
\end{center}
\maketitle
\begin{scriptsize}
\begin{tiny}
\begin{abstract}
Following attempts at an analytic proof of the Pentagonal Number
Theorem, we report on the discovery of a general principle leading to
an unexpected cancellation of oscillating sums, of which $\sum_{n^2\leq x}(-1)^ne^{\sqrt{x-n^2}}$
is an example (to get the idea of the result). It turns out that sums in the
class we consider are much smaller than would be predicted by certain
probabilistic heuristics. After stating the motivation, and our theorem,
we apply it to prove several results on the Prouhet-Tarry-Escott Problem, integer partitions, and the distribution
of prime numbers. Regarding the Prouhet-Tarry-Escott problem, for example, we show that
\begin{align*}
\sum_{|\ell|\leq x}(4x^2-4\ell^2)^{2r}-\sum_{|\ell|<x}(4x^2-(2\ell+1)^2)^{2r}=\text{polynomial w.r.t. } x \text{ with degree }2r-1.
\end{align*}
Note that the degree is unexpectedly small. This can perhaps be proved using properties of Bernoulli polynomials, but the claim fell out of our method in a more natural and motivated way. Using this result, we solve an approximate version of the Prouhet-Tarry-Escott Problem, and in doing so we give some evidence that a certain pigeonhole argument for solving the exact version of the Problem can be improved. In fact, our work in the approximate case exceeds the bounds one can prove using a pigeonhole argument, which seems remarkable. Also, we prove
that
$$
\sum_{\ell^2 < n} (-1)^\ell p(n-\ell^2)\ \sim\ (-1)^n 2^{-3/4} n^{-1/4} \sqrt{p(n)},
$$
where $p(n)$ is the usual partition function.
We get the following ``Weak pentagonal number theorem", in which we can replace the partition function $p(n)$ with Chebyshev $\Psi$ function:
$$
\sum_{0 < \ell < \sqrt{xT}/2} \Psi([e^{\sqrt{x - \frac{(2\ell)^2}{T}}},\
e^{\sqrt{x - \frac{(2\ell-1)^2}{T}}}])\ =\Psi(e^{\sqrt{x}})\left(\frac{1}{2} +
O\left (e^{-0.196\sqrt{x}}\right)\right),
$$
where $T=e^{0.786\sqrt{x}}$,
where $\Psi([a,b]) := \sum_{n\in [a,b]} \Lambda(n)$ and
$\Psi(x) = \Psi([1,x])$, where $\Lambda$ is the von
Mangoldt function. It is possible to extend this to a range of values of $T$, with a more complicated error term. Note that this last equation (sum over $\ell$) is
stronger than one would
get using a strong form of the Prime Number Theorem and also a naive use of the Riemann
Hypothesis in each interval,
since the widths of the intervals are smaller than $e^{\frac{1}{2}
\sqrt{x}}$, making the Riemann Hypothesis estimate ``trivial".
\end{abstract}

\end{tiny}
\end{scriptsize}


\section{Introduction}

In this paper, we use the notation $f(x)\simeq g(x)$, which means that there exists $C>0$ such that
$$\limsup_{x\rightarrow \infty} \frac{|f(x)|}{|g(x)|}=C;$$
In particular, we show $f\sim g$ if $C=1$. We sometimes need to use big Oh notation $f(x)=O(g(x))$, which means that
$$\limsup_{x\rightarrow \infty} \frac{|f(x)|}{|g(x)|} <\infty.$$
Also we say $f(x) = o(g(x))$ if
$$\liminf_{x\rightarrow \infty} \frac{|f(x)|}{|g(x)|} =0 .$$
and $f(x) = \omega(g(x))$ if $g(x) = o(f(x))$. Finally we may like to emphasize that the constants in these notations are functions of specific variable like $y$. In that case we write them like $\simeq_y$, $\ll_y$, $O_y(f(x))$, ...

\bigskip

The partition function $p(n)$ is the number of ways to represent $n$ as sum of positive non-decreasing integers.
The Pentagonal Number Theorem of Euler asserts that for an integer $x\geq 2$,
\begin{align*}
\sum_{G_n\leq x}(-1)^np(x-G_n)=0
\end{align*}
where $G_n=\frac{n(3n-1)}{2}$ is the $n$th pentagonal number. Various proofs of this theorem have been developed over the
decades and centuries - see \cite{berndt}; but we wondered whether it was possible to produce an ``analytic proof", using the Ramanujan-Hardy-Rademacher formula - see \cite{andrew} - for $p(x)$:
\begin{align}\label{RHR}
p(n)=\frac{1}{2\pi \sqrt{2}}\sum_{k=1}^{\infty}\left(\sqrt{k}\left(\sum_{\substack{0\leq h<k\\(h,k)=1}}\omega(h,k)e^{-\frac{2\pi ihn}{k}}\right)\frac{d}{dx}\left(\frac{\sinh\left(\frac{\pi}{k}\sqrt{\frac{2}{3}(x-\frac{1}{24})}\right)}{\sqrt{x-\frac{1}{24}}}\right)\Bigg\vert_{x=n}\right),
\end{align}
where $\omega(h,k)$ is a sum over some roots of unity. Considering just the first two terms in this formula, one sees that (see \cite{desalvo})
\begin{align*}
p(x)=p_2(x)+O(\sqrt{p(x)}), \quad \text{where} \quad p_2(x)=\frac{\sqrt{12}e^{\frac{\pi}{6}\sqrt{24x-1}}}{24x-1}\left(1-\frac{6}{\pi\sqrt{24x-1}}\right).
\end{align*}
So, the Pentagonal Number Theorem implies
\begin{align}\label{pentagonalapproximation}
\sum_{G_n\leq x}(-1)^np_2(x-G_n)\ll \sqrt{xp(x)};
\end{align}
In fact, one can get a better bound by using more terms in the Ramanujan-Hardy-Rademacher expression; one might call this a ``Weak Pentagonal Number Theorem", which is an
interesting and non-trivial bound for the size of this oscillating sum of
exponential functions $(-1)^np_2(x-G_n)$.
It is worth pointing out that this bound is much smaller than what
would be expected on probabilistic grounds: if we consider a sum
\begin{align*}
S(X_1,X_2,\cdots)=\sum_{G_n<x}X_np_2(x-G_n),
\end{align*}
where the $X_n$s are independent random variables taking the values $+1$
and $-1$, each with probability $\frac{1}{2}$, then $E(S^2)=\sum_{G_n\leq x}p_2(x-G_n)^2$. So the quality of bound we would expect to prove is
\begin{align*}
|S|\ll \left(\sum_{G_n\leq x}p_2(x-G_n)^2\right)^{\frac{1}{2}}\ll \sqrt[4]{x} p(x),
\end{align*}
However, the bound \eqref{pentagonalapproximation} is much smaller than the RHS here.

\bigskip

What we have discovered is that \eqref{pentagonalapproximation} is just the tip of the iceberg, and
that there is a very general class of sums like this that are small - much
smaller than one would guess based on a probabilistic heuristic. Roughly,
we will prove that
\begin{align}\label{equationexponential}
\sum_{f(n)\leq x} (-1)^n e^{c\sqrt{x-f(n)}}=\textrm{``small"},
\end{align}
where $f$ is a quadratic polynomial (with positive leading coefficient), and $c$ is some constant. It is possible to produce a more general class of sums with a lot of cancellation; and we leave it to the reader to explore. As a consequence of this and the Ramanujan-Hardy-Rademacher
expansion for $p(n)$, we will prove that
\begin{equation} \label{pentagonal_l2}
\sum_{l^2<n}(-1)^l p(n-l^2)\ \sim\ \frac{(-1)^n \sqrt{p(n)}}{2^{3/4} n^{1/4}}.
\end{equation}
As another category of results, we will also prove a corollary of Theorem \ref{psiresult1} related to prime numbers. In fact let $x>0$ be large enough and
$T=e^{0.786\sqrt{x}}$.
Then
$$
\sum_{0 \leq \ell < \frac{1}{2}\sqrt{xT}} \Psi\left([e^{\sqrt{x - \frac{(2\ell)^2}{T}}},\
e^{\sqrt{x - \frac{(2\ell-1)^2}{T}}}]\right)\ =\Psi(e^{\sqrt{x}})\left(\frac{1}{2} +
O\left (e^{-0.196\sqrt{x}}\right)\right).
$$
Finally we will develop polynomial identities that occur naturally in the Taylor expansion in \eqref{equationexponential}. For example
\begin{align*}
\sum_{\substack{|\ell|\leq x}}(4x^2-4\ell^2)^{2r}-\sum_{\substack{|\ell|<x}}(4x^2-(2\ell+1)^2)^{2r}=\text{polynomial w.r.t. } x \text{ with degree }2r-1.
\end{align*}

\bigskip

Many of the results stated above can be deduced from the following:


\begin{theorem}\label{main1}
Let $b,d\in\mathbb{R}$, $a,c>0$;
Also, let $h(x)$ be a polynomial with no 
real pole $p_0$ with $|p_0|<x^{\frac{1}{2}-\epsilon}$. 
Then
\begin{align}\label{pentagonalrealcasegeneral}
\sum_{n:an^2+bn+d<x}(-1)^n e^{c\sqrt{x-(an^2+bn+d)}}h(n)\ll e^{(w+\epsilon)c\sqrt{x}}.
\end{align}
where $w>0$ is defined as follows. Set
\begin{align*}
\Delta:=\sup_{r\geq 0}\sqrt{ \sqrt{a}r\frac{\sqrt{ar^2+4}+r\sqrt{a}}{2}}-\frac{\pi r}{c}. 
\end{align*}
Then $w=\min(1,\Delta)$.
\end{theorem}



\begin{remark}
Obviously forcing $w$ to be less than or equal to one is to avoid getting a trivial result, and if $a,c,r$ are chosen in such a way that $\sup_r \Delta(r)>1$,
then this theorem becomes useless.
\end{remark}


\bigskip


\begin{conjecture}
Observing the numerical results suggest that
\begin{align*}
\sum_{\ell^2<x}(-1)^{\ell} e^{\sqrt{x-\ell^2}}=e^{o(\sqrt{x})}.
\end{align*}
\end{conjecture}


\bigskip

There is another generalization when we pick a complex $c$ in \eqref{pentagonalrealcasegeneral}. In this case, having an upper bound for the sum is harder, as we have both the fast growth of exponential functions and the extra oscillation coming from the imaginary exponent.


\begin{theorem}\label{main2}
For large enough $x>0$, let $T:=T(x)$ be at least $\omega(x^2)$ as $x\rightarrow \infty$. Also let $\alpha+i\beta\in\mathbb{C}$ and $0\leq\alpha<1+\epsilon$ for a fixed $\epsilon>0$, and $|\beta|<\sqrt{T}$. Then for arbitrary $\delta>0$
\begin{align*}
\sum_{l^2<Tx}(-1)^l e^{(\alpha+i\beta)\sqrt{x-\frac{l^2}{T}}}
\ll \sqrt{\frac{Tx}{|\beta|+1}}e^{\alpha(\sqrt{\frac{1}{1+\pi^2}}+\delta)\sqrt{x}}+\sqrt{T}.
\end{align*}
\end{theorem}


Note that if $\beta=0$ and $T$ sufficiently large, theorem \ref{main2} becomes a special case of theorem \ref{main1} for $a=1$, $b,d=0$, and $c\rightarrow 0$ with a weaker result.

Even these theorems do not exhaust all the cancellation types of oscillatory sums of this
form, for we can replace the square-root by a fourth-
root, and then replace the quadratic polynomial with a quartic. We will
not bother to develop the most general theorem possible here.
Next, we prove three applications for these oscillation sums.

\bigskip

\bigskip


\subsection{Applications to the Chebyshev $\Psi$ function}

We show that in the ``Weak pentagonal number theorem" we can replace the partition function $p(n)$ with Chebyshev $\Psi$ function.


\begin{theorem}\label{psiresult1}
Assume $\epsilon>0$, $x$ is large enough and $T=e^{\frac{4}{3}(1-\frac{1}{\sqrt{1+\pi^2}})\sqrt{x}}$, then 
we have
\begin{align}\label{psipentagonalequation}
\sum_{|\ell|<\sqrt{Tx}}(-1)^{\ell} \Psi\left(e^{\sqrt{x-\ell^2/T}}\right)\ll e^{(\frac{2}{3}+\frac{1}{3\sqrt{1+\pi^2}}+\epsilon)\sqrt{x}}:= e^{w\sqrt{x}}.
\end{align}
\end{theorem}


\bigskip

We give an argument to show a relation between the theorem and the distribution of prime numbers. A weak version of theorem can be written as
\begin{align*}
\Psi(e^{\sqrt{x}})=2\sum_{0<\ell<\sqrt{xT}/2} \left(\Psi\left(e^{\sqrt{x-\frac{(2\ell -1)^2}{T}}}\right)-\Psi\left(e^{\sqrt{x-\frac{(2\ell)^2}{T}}}\right)\right)+O\left(e^{(\frac{7}{9}+\epsilon)\sqrt{x}}\right) \quad \text{where } T:= e^{\frac{8\sqrt{x}}{9}}.
\end{align*}
Consider the following set of intervals:
\begin{align*}
I:= \bigcup_{0<\ell<\sqrt{xT}/2}\left(e^{\sqrt{x-\frac{(2\ell)^2}{T}}},e^{\sqrt{x-\frac{(2\ell-1)^2}{T}}}\right)
\end{align*}
One can see that the measure of $I$ is almost half of the length of the interval $[0,e^{\sqrt{x}}]$. Roughly speaking theorem \ref{psiresult1} states that the number of primes in $I$, with weight $\log(p)$, is half of the number of primes, with the same weight. This prime counting gives a stronger result than one would
get using a strong form of the Prime Number Theorem and also the Riemann
Hypothesis(RH), where one naively estimates the $\Psi$ function on each of the
intervals. Because the widths of the intervals are smaller than $e^{\frac{\sqrt{x}}{2}
}$, making the Riemann Hypothesis estimate ``trivial". However, a less naive argument can give an improvement like corollary \ref{lem:RHassumption}. See Table \ref{Prime} for comparison.


\begin{table}[h!]
\begin{center}
\begin{tabular}{ |c |c| c| c| c | }
\hline
 \text{ PNT} & \text{Naive RH + Theorem \ref{main1}} & \text{Theorem \ref{psiresult1} unconditionally} & \text{Corollary \ref{lem:RHassumption} with RH} \\
\hline
  1.46 & 0.96 & 0.78 & 0.38 \\
\hline
\end{tabular}
\caption{The upper bound of $w$ in equation \eqref{psipentagonalequation}}
\end{center}
\label{Prime}
\end{table}



\begin{corollary}\label{lem:RHassumption}
Assuming RH and the same notation as in theorem \ref{psiresult1}
\begin{align}\label{RHassumption}
\sum_{|\ell|<\sqrt{xT}}(-1)^{\ell} \Psi\left(e^{\sqrt{x-\ell^2/T}}\right)\lesssim e^{(\frac{1}{3}+\frac{1}{6\sqrt{1+\pi^2}})\sqrt{x}}\lesssim e^{0.38\sqrt{x}}.
\end{align}
\end{corollary}


The proof needs careful computations of a cancellation sum involving zeroes of the Riemann zeta function. In fact, we use our cancellation formula to control the low-height zeroes; The Van der Corput bound for exponential sums combined with the Montgomery Mean-value theorem to control the high-height zeroes.


\begin{remark}
Note that numerical results up to $x<300$ show a very smaller error term in comparison to \eqref{RHassumption}. In particular, for example,
\begin{align*}
\left\vert\sum_{l<2400}(-1)^l\Psi(e^{\sqrt{300-l^2/T}})\right\vert < 100 \quad \text{ where }T\sim 20000.
\end{align*}
\end{remark}



\begin{remark}
A more applicable identity may be one having with fewer terms (with lower frequency) in \eqref{psipentagonalequation}. We can choose the parameters to get
\begin{align*}
\sum_{l^2<xe^{2\epsilon\sqrt{x}}}(-1)^l\Psi(e^{\sqrt{x-l^2e^{-2\epsilon\sqrt{x}}}})\lesssim x^2e^{(1-\epsilon)\sqrt{x}}.
\end{align*}
This identity does not give the same level of cancellation as RH anymore but still is better than the best cancellation one can get from the current unconditional estimates for $\Psi$ function. Also, the advantage is that the intervals $\left(e^{\sqrt{x-(2\ell)^2\epsilon\sqrt{x}}},e^{\sqrt{x-(2\ell-1)^2\epsilon\sqrt{x}}}\right)$ are not as small as what we had in \eqref{psipentagonalequation}. So it possibly is more suitable for combinatorial applications.
\end{remark}


\bigskip


\subsection{Applications to the usual and restricted partitions}
A generalization of the Pentagonal Number Theorem is the second application of the cancellation result. It is an interesting question to find the second dominant term of general,``Meinardus type" integer partitions. Our result is applicable in general if the second term of Meinardus's Theorem is known. But the known asymptotic formulas rely heavily on analytic properties of the parts. For many partition functions $\lambda(n)$, we see a formula like 
\begin{align}\label{1}
\lambda(n)\sim \left(g(n)\right)^qe^{\left(k(n)\right)^{\theta}}\left(1-\frac{1}{(h(n))^r}\right)+O(\lambda(n)^{s})
\end{align}
where $0<s<1$ and $\theta,r,q>0$ and $k(n)$ is a linear polynomial and $g(n),h(n)$ are rational functions. For example for the usual partition function we have \begin{align*}
g(n)=\frac{\sqrt{12}}{24n-1}\quad ,\quad
h(n)=\frac{\pi^2}{36}(24n-1)\quad , \quad
k(n)=\frac{\pi^2}{36}(24n-1) \quad ,\quad
s=\theta=\frac{q}{2}=r=\frac{1}{2}
\end{align*}
Assuming a partition function has form \eqref{1}, we can conclude that for a quadratic polynomial $t(n)=an^2+bn+d$ as stated in theorem \ref{main1}  we have
\begin{align*}
\sum_{\ell: t(\ell)<n}(-1)^l \lambda(n-t(\ell)) \ll \lambda^{\kappa}(n)
\end{align*}
where $\kappa=\max(w,s)$ and $w$ is defined as in Theorem \ref{main1}, and $s$ in \eqref{1}. As long as $\kappa<1$, we can get a nontrivial approximation of the Pentagonal Number Theorem. We will verify this inequality for some polynomials and give a few specific examples.

\bigskip

First, we mention a weak pentagonal number theorem for certain approximations of the partition function.


\begin{proposition}\label{pentcorollary}
Let
\begin{align*}
p_1(x)&=\frac{e^{\pi\sqrt{\frac{2x}{3}}}}{4x\sqrt{3}}\\
p_2(x)&=(\frac{\sqrt{12}}{24x-1}-\frac{6\sqrt{12}}{\pi(24x-1)^{\frac{3}{2}}})e^{\frac{\pi}{6}\sqrt{24x-1}}\\
p_3(x)&=(\frac{\sqrt{6}(-1)^x}{24x-1}-\frac{12\sqrt{6}(-1)^x}{\pi(24x-1)^{\frac{3}{2}}})e^{\frac{\pi}{12}\sqrt{24x-1}}
\end{align*}
be the asymptotic approximation of partition function, the contribution for $k=1$, and the contribution for $k=2$ in  the Ramanujan-Hardy-Rademacher formula \eqref{RHR}, respectively. Then
\begin{align}
&\sum_{G_l<x}(-1)^{l}p_1(x-G_l)\ll p(x)^{0.31}\label{nonsensitive}\\
&\sum_{G_l<x}(-1)^lp_2(x-G_l)\ll p(x)^{0.21} \label{sensitive}\\
&\sum_{G_l<x}(-1)^l\sqrt{p_1(x-G_l)}\ll p(x)^{0.13}\label{pentuu}\\
&\sum_{l^2<x}p_3(x-l^2)\ll p(x)^{0.11} \label{thirdfourthestimate}
\end{align}
\end{proposition}


We mentioned  \eqref{pentuu} here because they have very small error terms.
\begin{remark}
Equation \eqref{sensitive} gives us a very promising perspective to possibly prove the pentagonal number theorem using an analytic argument. We will state a possible future direction after giving its proof.
\end{remark}

Note that equation \eqref{thirdfourthestimate} does not have the factor $(-1)^l$, because $\sum_h \omega(h,2)$ in equation \eqref{RHR} is $\frac{(-1)^l}{\sqrt{2}}$. So it can cancels out the other $(-1)^l$ from the weak pentagonal number theorem to eliminate the cancellation. In fact, if we put $(-1)^l$, we get the following proposition.

\bigskip


\begin{proposition}\label{pack}
For large enough integer $x$
\begin{align*}
\sum_{l^2<x}(-1)^lp_3(x-l^2)\sim \frac{e^{\pi ix}\sqrt{p(x)}}{\sqrt[4]{8x}}.
\end{align*}
So if $p_4(x)=p_2(x)+p_3(x)$ is the first ``four" terms in the Ramanujan-Hardy-Rademacher expression for the partition function, then we get
\begin{align}\label{fulpack}
\sum_{l^2<x}(-1)^lp_4(x-l^2)\sim\frac{\sqrt{p(x)}}{\sqrt[4]{8x}},
\end{align}
which immediately implies (\ref{pentagonal_l2}).
\end{proposition}


\bigskip

We mention another set of examples. It is reported in \cite[Theorem 4]{iseki} that $p(n;\alpha, M)$ the number of partitions with parts of the form $Mt\pm \alpha$, $1\leq \alpha\leq M-1$, and $(\alpha,M)=1$ is
\begin{align*}
p(n;\alpha, M)=\frac{\pi\csc(\frac{\pi \alpha}{M})}{\sqrt{12Mn-6\alpha^2+6M\alpha-M^2}}
I_1\left(\frac{\pi\sqrt{12Mn-6\alpha^2+6M\alpha-M^2}}{3M}\right)+O(e^{\frac{\pi\sqrt{n}}{\sqrt{3M}}}).
\end{align*}
Theorem \ref{main1} can show a weak pentagonal number expression like
\begin{align*}
\sum_{am^2+bm+d<x}(-1)^m p\left(x-am^2-bm-d;\alpha, M\right)\simeq\text{``small function w.r.t. }x,\alpha, a ,M\text{"}.
\end{align*}
where $I_1$ is the modified Bessel function of the first kind.  We take two cases $M=2$ and $M=5$ as examples. It is known that number of partitions into distinct parts, see \cite{hao}, is
\begin{align*}
q(n):= A\frac{d}{dx}\left(I_0\left(\pi \sqrt{\frac{2x+\frac{1}{12}}{3}}\right)\right)\Bigg\vert_{x=n}+O(\sqrt{q(n)}).
\end{align*}
where $A$ is a constant and $I_0$ is Bessel function.


\begin{corollary}\label{penbessel1}
Let $q_1(n)$ be the first two terms in the expansion of $q(n)$. For large $n$
\begin{align*}
\sum_{l^2\leq n}(-1)^l q_1(n-l^2)&\ll q(n)^{0.151}\\
\sum_{l^2\leq n}(-1)^l q(n-l^2)&\ll \sqrt{q(n)}.
\end{align*}
\end{corollary}


Also for $M=5$, see \cite{lehner}, there exists a constant $A:=A_a>0$ and a constant $B>0$ such that
\begin{align*}
p(n;a,5)\sim \frac{B\pi\csc(\frac{\pi a}{5})}{(60n-A)^{\frac{3}{4}}}
e^{\frac{\pi\sqrt{60n-A}}{15}}
\end{align*}


\begin{corollary}\label{penbessel2}
Let $h(n)$ be the main contribution  in the expansion of $p(n;a,5)$. For large $n$
\begin{align*}
&\sum_{l^2\leq n}(-1)^l h(n-l^2)\ll p(n;a,5)^{0.13}\\
&\sum_{\ell^2\leq n}(-1)^{\ell} p(n-\ell^2;a,5)\ll \sqrt{p(n;a,5)}.
\end{align*}
\end{corollary}



\bigskip


\subsection{Applications to the Prouhet-Tarry-Escott Problem}

Another application of our method involves the so-called Prouhet-Tarry-Escott Problem. The problem is to determine, for a fixed integer $n \geq 1$,
the largest value $k$ (denote by $M(n)$) with $k \leq n-1$ for which there exist two sequences of $n$
integers $ a_1, \cdots, a_n$ and $b_1,\cdots, b_n$, say - such that for all $1\leq r\leq k$ we have (see \cite{Raghanderavan})
\begin{align}\label{terryescott}
\sum_{i=1}^na_i^r=&\sum_{i=1}^nb_i^r \nonumber\\
\sum_{i=1}^na_i^{k+1}\neq &\sum_{i=1}^nb_i^{k+1} \text{ and multi-sets } \{a_i\},\{b_i\} \textup{ are disjoint}. 
\end{align}

\bigskip

One could consider a weakening of this problem, where the left and right
hand sides of \eqref{terryescott} are merely required to be ``close to each other". One way to naturally view this approximation is to interpret $\{a_i\},\{b_i\}$ to be values taken on by two  discrete uniform random variables $A,B$ that takes on rational values and both of whose moments (up to a certain level) and their moment generating functions are ``close"; i.e. the probability density function of these random variables becomes almost the same.
Approximating moment generating functions is an important problem in the literature - see for example \cite{butler,xu}; and what we are interested in is that the probability space is a subset of $\mathbb{Q}$. This makes the problem non-trivial.

\medskip

This problem can be also viewed from another perspective that is related to the Vingradov mean value theorem. Consider $J_{n,k}(N)$ to be the number of solutions in equations \eqref{terryescott} and remove the restriction $a_i\neq b_j$ and assume that $1\leq a_i,b_i\leq N$. It was conjectured that
\begin{equation}\label{bourgain}
J_{n,k}(N) \ll_{\epsilon} N^{\epsilon}\left(N^n + N^{2n - \frac{1}{2}k(k+1)}\right).
\end{equation}
Note that there are $N^n$ trivial solutions in this setting. This conjecture has been proved for $k=3$ by Wooley at \cite{wooley16} and was proved for $k\geq 4$ at \cite{Bourgain16} by Bourgain, Demeter, and Guth. Note that this view is closely related to Waring's problem. For more information about this direction, please check the survey paper \cite{Pierce}. 

\bigskip


Let us formulate the problem as follows.
\begin{problem}\label{terryescottapproximation}
Let $0 < c = c(N,n,k) < 1$ be the smallest constant such that there
exist sequences of integers
\begin{align*}
1\leq a_1\leq a_2\leq \cdots\leq a_n\leq N \quad\text{ and }\quad 1\leq b_1\leq b_2\leq \cdots\leq b_n\leq N
\end{align*}
that do not overlap  and  for all $1\leq r\leq k $,
\begin{align}\label{appPTEproblem}
\left\vert\sum_{i=1}^{n}a_i^r-b_i^r\right\vert\leq N^{cr}
\end{align}
How small can we take $c$ to be for various ranges of $k$ and $n$?
\end{problem}


There has been little progress in solving the original PTE problem. For example for an ideal solution (when $k=n-1$) the largest known solution is for $n=12$, see \cite{Raghanderavan}. To our knowledge, the best constructive solution is perhaps for the range $k=O(\log(n))$. Using a pigeonhole argument we can do much better, and give non-constructive solutions with $k$ as large as $k \sim c \sqrt{n}$. In section \ref{PTE} we will briefly explain this argument, which gives one of the best known non-constructive ways to solve the problem \ref{terryescottapproximation}.

Note that using the Vingradov conjecture \eqref{bourgain}, the trivial upper bound becomes sharper when $k>2\sqrt{n}$ (the right hand side becomes $N^{n+\epsilon}$). So we expect to have harder time finding a nontrivial solution.
Even applying the pigeonhole argument to the approximate version (Problem \ref{terryescottapproximation}) we cannot make $k$ much larger; for example, we cannot prove the existence of non-decreasing sequences $a_i$ and $b_i$ such that
\begin{align*}
|\sum_i a_i^r - b_i^r| < N^{r(1-\frac{1}{\log n})} \quad\text{ for all }\quad 1 \leq r \leq \sqrt{n}(\log n)^2.
\end{align*}
In other word, one cannot guarantee that the value of $c$ in Problem \ref{terryescottapproximation} can be $c < 1 - 1/\log n$ when $k > \sqrt{n}(\log n)^2$. We will see that this range for $c$ is much, much weaker than what our construction gives. This suggests that it might be possible to beat the bounds that the pigeonhole principle gives for the exact version of the problem.
 We will give a proof to the following theorem, which states a \textit{constructive} solution for problem \ref{terryescottapproximation} when $k$ is much bigger than $\sqrt{n}$.


\begin{theorem}
Equation \eqref{appPTEproblem} has a constructive solution for $c\sim  \frac{13}{28}$  and $k \sim \frac{n^{6/7}}{\log(n)}$ as follows. For $1\leq i\leq \sqrt{N}$
\begin{align*}
a_i = N-(2i-2)^2\in \mathbb{N}\quad \text{ and }\quad
b_i=N-(2i-1)^2\in \mathbb{N} .
\end{align*}
Then for all $1\leq r\leq k\sim N^{3/7}$ we have
\begin{align}\label{pte1}
\sum_{1\leq i\leq n}a_i^r-\sum_{1\leq i\leq n} b_i^r\ll r^{r/2} N^{\frac{r+1}{4}} \ll N^{\frac{r+1}{4}+ \frac{r\log(r)}{2\log(N)}}.
\end{align}
\end{theorem}


In fact we prove a more general theorem In section \ref{PTE} as follows.


\begin{theorem}\label{pte0}
Let $L\geq 1$ and $m\in\mathbb{N}$ and define $M=\lfloor (2L)^{\frac{2m}{2m+1}}\rfloor$. Define for $1\leq i\leq n:=L^m$
\begin{align*}
a_i=M^{2m+1}-(2i-2)^2 \quad,
b_i=M^{2m+1}-(2i-1)^2 \quad , 1\leq r\leq k\ll_m \frac{M^m}{\log(M)}.
\end{align*}
Then
\begin{align}\label{pte1}
\sum_{1\leq i\leq L}a_i^r-\sum_{1\leq i\leq L} b_i^r\ll r^{\frac{r}{2}+\epsilon}M^{(2m+1)\frac{r}{4}+m} \ll r^{\frac{r}{2}+\epsilon} a_1^{\frac{r+1}{4}}.
\end{align}
So we have two sets of around $n$ integers less than $N:=M^{2m+1}$, and they are satisfying the equation \eqref{appPTEproblem} with $c\sim  \frac{1}{2}-\frac{1}{4m+2}+\epsilon$ and $k \simeq n^{1-\frac{1}{m}}$.
\end{theorem}



\begin{remark}
There is a conjecture in \cite{chen} stating that if $\{a_n\geq 0\},\{b_n\geq 0\}$ be an ideal solution of Prouhet-Tarry-Escott problem and $a_1<b_1$, then for all $i$
\begin{align}\label{conjchen}
(a_i-b_i)(a_{i+1}-b_{i+1})<0.
\end{align}
Although our example cannot resolve the conjecture, it shows that equation \eqref{conjchen} is not true for the solutions of Problem \ref{terryescottapproximation} for any $c$.
\end{remark}



\begin{remark}
Note that we can win by a constant factor - i.e. increase $M(n)$ by a constant, if we pick a suitable quadratic polynomial $q(l)$ instead of $l^2$.
\end{remark}


\bigskip

Lastly, we investigate the problem more concretely by viewing $a_i,b_i$ as polynomials. Then this cancellation sum can be considered as an operator in $\mathbb{Z}[x]$ which cuts the degree to half.


\begin{theorem}\label{polynomialcase}
Let $M \in \mathbb{N}$, and define $f_r(M) := \sum_{|\ell| < 2M} (-1)^{\ell} (4M^2 - \ell^2)^r$; Then, $f_r(M)$ is a polynomial of degree $r-1$ in $M$ when $r$ is even, and is a polynomial of degree $r$ in $M$ when $r$ is odd; that is, when $r$ is even,
$$ f_r(M) = c_0(r) + c_1(r) M + \cdots + c_{r-1}(r) M^{r-1}, $$
where $c_0(r),...,c_{r-1}(r)$ are integer functions of $r$ only (and not of $M$). The same general form holds for $r$ odd, except that the degree here is $r$, not $r-1$. Furthermore, under the assumption $r \ll \frac{M}{\log(M)}$ we have that all the coefficients have size $O(r^{r+\epsilon})$.
\end{theorem}


\bigskip


\section{Proof of the oscillation sums}

In this section we mainly prove theorems \ref{main1} and \ref{main2}. First, we mention a lemma.


\begin{lemma}\label{l:sqrt}
Let $z= A+iB$ be a complex number and $q(s)=as^2+bs+c$ be a quadratic polynomial and $x\in \mathbb{R}$. Then
\begin{align}
\text{Re}\left(\sqrt{x-q(z)} \right)= \sqrt{\frac{1}{2}\left(\sqrt{D^2+\left(2aAB + bB\right)^2} +D\right)}\label{real}\\
\text{Im}\left(\sqrt{x-q(z)} \right)= \sqrt{\frac{1}{2}\left(\sqrt{D^2+\left(2aAB + bB\right)^2} -D\right)}\label{imaginary}.
\end{align}
where
\begin{align*}
D := x - (A^2-B^2)a - bA -c.
\end{align*}
\end{lemma}



\begin{proof}
We only prove \eqref{real}.
We have
$$x-q(z) = x-(A^2-B^2)a-bA-c - i (2ABa+bB).$$
It impliess that
$$\text{Re}\left(\sqrt{x-q(z)} \right)= \sqrt[4]{D^2+\left(2aAB + bB\right)^2} \cos\left(\frac{1}{2}\tan^{-1}\left(\frac{2aAB + bB}{D}\right)\right).$$
Noting that $\cos^2(y)=\frac{1+\cos(2y)}{2}$ and $\cos(\arctan(y))=\frac{1}{\sqrt{1+y^2}}$
$$\text{Re}\left(\sqrt{x-q(z)} \right) = \sqrt[4]{D^2+\left(2aAB + bB\right)^2} \sqrt{\frac{1}{2} + \frac{1}{2\sqrt{1+\frac{(2aAB + bB)^2}{D^2}}}}.$$
Straightforward computation results in equation \eqref{real}.
\end{proof}


\bigskip


Next, we prove Theorem \ref{main1}.
\begin{proof}
Let $q(z):=az^2+bz+d$ and $f(z)=\sqrt{x-q(z)}$ with branch points $\alpha_1,\alpha_2$.
We choose $(-\infty,\alpha_1]\cup [\alpha_2,\infty)$ as the branch cut and let $G$ be the interior of the square with vertices
\begin{align*}
\pm(\sqrt{\frac{x}{a}}-\frac{2b}{a})\pm iu\sqrt{x},
\end{align*}
where $u>0$ will be chosen later. Note that $$\alpha_2>\sqrt{\frac{x}{a}}-\frac{2b}{a}$$
and we have a similar condition for $\alpha_1$. Without loss of generality we assume that $h$ is holomorphic inside $G$.


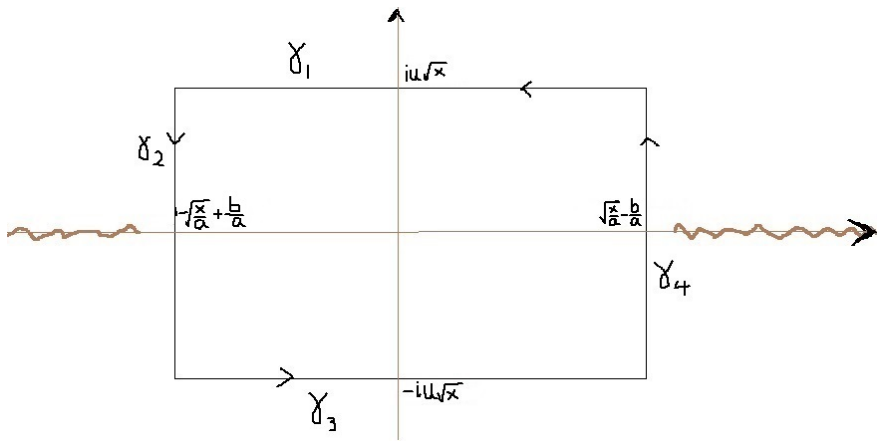
\begin{figure}[h!]

\tikzset{every picture/.style={line width=0.75pt}} 

\begin{tikzpicture}[x=0.75pt,y=0.75pt,yscale=-1,xscale=1]

\draw [line width=1.5] (270.8,1041.69) -- (508.8,1041.69)(294.6,857.19) -- (294.6,1062.19) (501.8,1036.69) -- (508.8,1041.69) -- (501.8,1046.69) (289.6,864.19) -- (294.6,857.19) -- (299.6,864.19) ;
\draw (176,929.69) -- (185,924.69) ;
\draw [line width=1.5] (54,1039.69) -- (271,1041.69) ;
\draw [line width=1.5] (294,1056.69) -- (296,1175.69) ;
\draw (130.1,930.19) -- (459.1,930.19) -- (459.1,1153.19) -- (130.1,1153.19) -- cycle ;
\draw [color={rgb, 255:red, 208; green, 2; blue, 27 } ,draw opacity=1 ][line width=1.5] (474,1041.69) -- (506,1040.69) ;
\draw [color={rgb, 255:red, 208; green, 2; blue, 27 } ,draw opacity=1 ][line width=1.5] (57,1039.69) -- (111,1039.69) ;
\draw [color={rgb, 255:red, 208; green, 2; blue, 27 } ,draw opacity=1 ][line width=3] [line join = round][line cap = round] (206,1037.69) .. controls (206,1038.02) and (206,1038.35) .. (206,1038.69) ;
\draw [color={rgb, 255:red, 208; green, 2; blue, 27 } ,draw opacity=1 ][line width=3] [line join = round][line cap = round] (213,1037.69) .. controls (213,1037.69) and (213,1037.69) .. (213,1037.69) ;
\draw [color={rgb, 255:red, 208; green, 2; blue, 27 } ,draw opacity=1 ][line width=3] [line join = round][line cap = round] (220,1037.7) .. controls (220,1037.7) and (220,1037.7) .. (220,1037.7) ;
\draw [color={rgb, 255:red, 208; green, 2; blue, 27 } ,draw opacity=1 ][line width=3] [line join = round][line cap = round] (355,1039.69) .. controls (355,1039.69) and (355,1039.69) .. (355,1039.69) ;
\draw [color={rgb, 255:red, 208; green, 2; blue, 27 } ,draw opacity=1 ][line width=3] [line join = round][line cap = round] (364,1038.69) .. controls (364,1039.02) and (364,1039.35) .. (364,1039.69) ;
\draw [color={rgb, 255:red, 208; green, 2; blue, 27 } ,draw opacity=1 ][line width=3] [line join = round][line cap = round] (371,1038.69) .. controls (371,1039.02) and (371,1039.35) .. (371,1039.69) ;
\draw (176,930.69) -- (184,935.69) ;
\draw (280,1152.69) -- (414,1153.67) ;
\draw [shift={(416,1153.69)}, rotate = 180.42] [color={rgb, 255:red, 0; green, 0; blue, 0 } ][line width=0.75] (10.93,-3.29) .. controls (6.95,-1.4) and (3.31,-0.3) .. (0,0) .. controls (3.31,0.3) and (6.95,1.4) .. (10.93,3.29) ;
\draw (493,1070.69) -- (465.54,1047.96) ;
\draw [shift={(464,1046.69)}, rotate = 399.61] [color={rgb, 255:red, 0; green, 0; blue, 0 } ][line width=0.75] (10.93,-3.29) .. controls (6.95,-1.4) and (3.31,-0.3) .. (0,0) .. controls (3.31,0.3) and (6.95,1.4) .. (10.93,3.29) ;
\draw (103,1004.69) -- (124.76,1032.12) ;
\draw [shift={(126,1033.69)}, rotate = 231.57999999999998] [color={rgb, 255:red, 0; green, 0; blue, 0 } ][line width=0.75] (10.93,-3.29) .. controls (6.95,-1.4) and (3.31,-0.3) .. (0,0) .. controls (3.31,0.3) and (6.95,1.4) .. (10.93,3.29) ;
\draw [color={rgb, 255:red, 208; green, 2; blue, 27 } ,draw opacity=1 ] (440,1041.7) ;
\draw [shift={(440,1041.7)}, rotate = 0] [color={rgb, 255:red, 208; green, 2; blue, 27 } ,draw opacity=1 ][fill={rgb, 255:red, 208; green, 2; blue, 27 } ,fill opacity=1 ][line width=0.75] (0, 0) circle [x radius= 3.35, y radius= 3.35] ;
\draw [color={rgb, 255:red, 208; green, 2; blue, 27 } ,draw opacity=1 ] (419,1041.7) ;
\draw [shift={(419,1041.7)}, rotate = 0] [color={rgb, 255:red, 208; green, 2; blue, 27 } ,draw opacity=1 ][fill={rgb, 255:red, 208; green, 2; blue, 27 } ,fill opacity=1 ][line width=0.75] (0, 0) circle [x radius= 3.35, y radius= 3.35] ;
\draw [color={rgb, 255:red, 208; green, 2; blue, 27 } ,draw opacity=1 ] (401,1041.7) ;
\draw [shift={(401,1041.7)}, rotate = 0] [color={rgb, 255:red, 208; green, 2; blue, 27 } ,draw opacity=1 ][fill={rgb, 255:red, 208; green, 2; blue, 27 } ,fill opacity=1 ][line width=0.75] (0, 0) circle [x radius= 3.35, y radius= 3.35] ;
\draw [color={rgb, 255:red, 208; green, 2; blue, 27 } ,draw opacity=1 ] (333,1041.7) ;
\draw [shift={(333,1041.7)}, rotate = 0] [color={rgb, 255:red, 208; green, 2; blue, 27 } ,draw opacity=1 ][fill={rgb, 255:red, 208; green, 2; blue, 27 } ,fill opacity=1 ][line width=0.75] (0, 0) circle [x radius= 3.35, y radius= 3.35] ;
\draw [color={rgb, 255:red, 208; green, 2; blue, 27 } ,draw opacity=1 ] (314,1041.7) ;
\draw [shift={(314,1041.7)}, rotate = 0] [color={rgb, 255:red, 208; green, 2; blue, 27 } ,draw opacity=1 ][fill={rgb, 255:red, 208; green, 2; blue, 27 } ,fill opacity=1 ][line width=0.75] (0, 0) circle [x radius= 3.35, y radius= 3.35] ;
\draw [color={rgb, 255:red, 208; green, 2; blue, 27 } ,draw opacity=1 ] (294.6,1041.69) ;
\draw [shift={(294.6,1041.69)}, rotate = 0] [color={rgb, 255:red, 208; green, 2; blue, 27 } ,draw opacity=1 ][fill={rgb, 255:red, 208; green, 2; blue, 27 } ,fill opacity=1 ][line width=0.75] (0, 0) circle [x radius= 3.35, y radius= 3.35] ;
\draw [color={rgb, 255:red, 208; green, 2; blue, 27 } ,draw opacity=1 ] (276,1041.7) ;
\draw [shift={(276,1041.7)}, rotate = 0] [color={rgb, 255:red, 208; green, 2; blue, 27 } ,draw opacity=1 ][fill={rgb, 255:red, 208; green, 2; blue, 27 } ,fill opacity=1 ][line width=0.75] (0, 0) circle [x radius= 3.35, y radius= 3.35] ;
\draw [color={rgb, 255:red, 208; green, 2; blue, 27 } ,draw opacity=1 ] (258,1041.7) ;
\draw [shift={(258,1041.7)}, rotate = 0] [color={rgb, 255:red, 208; green, 2; blue, 27 } ,draw opacity=1 ][fill={rgb, 255:red, 208; green, 2; blue, 27 } ,fill opacity=1 ][line width=0.75] (0, 0) circle [x radius= 3.35, y radius= 3.35] ;
\draw [color={rgb, 255:red, 208; green, 2; blue, 27 } ,draw opacity=1 ] (240,1041.7) ;
\draw [shift={(240,1041.7)}, rotate = 0] [color={rgb, 255:red, 208; green, 2; blue, 27 } ,draw opacity=1 ][fill={rgb, 255:red, 208; green, 2; blue, 27 } ,fill opacity=1 ][line width=0.75] (0, 0) circle [x radius= 3.35, y radius= 3.35] ;
\draw [color={rgb, 255:red, 208; green, 2; blue, 27 } ,draw opacity=1 ] (188,1040.7) ;
\draw [shift={(188,1040.7)}, rotate = 0] [color={rgb, 255:red, 208; green, 2; blue, 27 } ,draw opacity=1 ][fill={rgb, 255:red, 208; green, 2; blue, 27 } ,fill opacity=1 ][line width=0.75] (0, 0) circle [x radius= 3.35, y radius= 3.35] ;
\draw [color={rgb, 255:red, 208; green, 2; blue, 27 } ,draw opacity=1 ] (168,1040.7) ;
\draw [shift={(168,1040.7)}, rotate = 0] [color={rgb, 255:red, 208; green, 2; blue, 27 } ,draw opacity=1 ][fill={rgb, 255:red, 208; green, 2; blue, 27 } ,fill opacity=1 ][line width=0.75] (0, 0) circle [x radius= 3.35, y radius= 3.35] ;
\draw [color={rgb, 255:red, 208; green, 2; blue, 27 } ,draw opacity=1 ] (147,1040.7) ;
\draw [shift={(147,1040.7)}, rotate = 0] [color={rgb, 255:red, 208; green, 2; blue, 27 } ,draw opacity=1 ][fill={rgb, 255:red, 208; green, 2; blue, 27 } ,fill opacity=1 ][line width=0.75] (0, 0) circle [x radius= 3.35, y radius= 3.35] ;
\draw (147,1040.7) -- (169.93,1077.01) ;
\draw [shift={(171,1078.7)}, rotate = 237.72] [color={rgb, 255:red, 0; green, 0; blue, 0 } ][line width=0.75] (10.93,-3.29) .. controls (6.95,-1.4) and (3.31,-0.3) .. (0,0) .. controls (3.31,0.3) and (6.95,1.4) .. (10.93,3.29) ;
\draw (168,1040.7) -- (177.44,1072.78) ;
\draw [shift={(178,1074.7)}, rotate = 253.61] [color={rgb, 255:red, 0; green, 0; blue, 0 } ][line width=0.75] (10.93,-3.29) .. controls (6.95,-1.4) and (3.31,-0.3) .. (0,0) .. controls (3.31,0.3) and (6.95,1.4) .. (10.93,3.29) ;
\draw (188,1040.7) -- (188.94,1072.7) ;
\draw [shift={(189,1074.7)}, rotate = 268.32] [color={rgb, 255:red, 0; green, 0; blue, 0 } ][line width=0.75] (10.93,-3.29) .. controls (6.95,-1.4) and (3.31,-0.3) .. (0,0) .. controls (3.31,0.3) and (6.95,1.4) .. (10.93,3.29) ;
\draw (440,1041.7) -- (209.96,1085.33) ;
\draw [shift={(208,1085.7)}, rotate = 349.26] [color={rgb, 255:red, 0; green, 0; blue, 0 } ][line width=0.75] (10.93,-3.29) .. controls (6.95,-1.4) and (3.31,-0.3) .. (0,0) .. controls (3.31,0.3) and (6.95,1.4) .. (10.93,3.29) ;
\draw (401,1041.7) -- (207.96,1078.33) ;
\draw [shift={(206,1078.7)}, rotate = 349.26] [color={rgb, 255:red, 0; green, 0; blue, 0 } ][line width=0.75] (10.93,-3.29) .. controls (6.95,-1.4) and (3.31,-0.3) .. (0,0) .. controls (3.31,0.3) and (6.95,1.4) .. (10.93,3.29) ;

\draw (166,1079.7) node [anchor=north west][inner sep=0.75pt] [align=left] {Poles};
\draw (9,957.69) node [anchor=north west][inner sep=0.75pt] [xslant=-0.04] [align=left] {$\displaystyle -\sqrt{\frac{x}{a}} +\sqrt{\frac{b}{a}}$};
\draw (486,1059.69) node [anchor=north west][inner sep=0.75pt] [xslant=-0.04] [align=left] {$\displaystyle \sqrt{\frac{x}{a}} -\sqrt{\frac{b}{a}}$};
\draw (296,1152.69) node [anchor=north west][inner sep=0.75pt] [align=left] {$\displaystyle -u\sqrt{x}$};
\draw (296,898.69) node [anchor=north west][inner sep=0.75pt] [align=left] {$\displaystyle u\sqrt{x}$};
\draw (464.75,943.89) node [anchor=north west][inner sep=0.75pt] [align=left] {$\displaystyle \gamma _{4}$};
\draw (411,1154.69) node [anchor=north west][inner sep=0.75pt] [align=left] {$\displaystyle \gamma _{3}$};
\draw (106,1108.69) node [anchor=north west][inner sep=0.75pt] [align=left] {$\displaystyle \gamma _{2}$};
\draw (137,903.69) node [anchor=north west][inner sep=0.75pt] [align=left] {$\displaystyle \gamma _{1}$};

\end{tikzpicture}

\centering \caption{The contour $\gamma$}
\label{shekl}
\end{figure}


Let
$g(z)=e^{c f(z)}$, which is analytic inside $G$. Define
\begin{align*}
H(z)=\frac{g(z)h(z)}{\sin(\pi z)}.
\end{align*}
Assume that $\gamma$ is the boundary of $G$ (see figure \ref{shekl}). Using the residue theorem
\begin{align}\label{pent}
\int_{\gamma}H(z)dz=2\pi i \sum_{z_j:\text{ poles}}\text{Res}(H(z))|_{z_j}=2\pi i\sum_{q(n)<x}(-1)^nh(n)e^{c\sqrt{x-q(n)}}.
\end{align}
We wish to show that the integral in LHS has size of at most $e^{cw\sqrt{x}}$. First assume that we choose $z\in \gamma_1\cup\gamma_3$. So $z=t\pm iu\sqrt{x}$ for $-\sqrt{\frac{x}{a}}+\frac{2b}{a}<t<\sqrt{\frac{x}{a}}-\frac{2b}{a}$. If $t=o(\sqrt{x})$, then $\sqrt{x-az^2-bz-d}\sim \sqrt{x(1+au^2)}$. Otherwise by lemma \ref{l:sqrt}
\begin{align}\label{reza}
Re(\sqrt{x-az^2-bz-d})\ll \sqrt{\frac{\sqrt{(x-at^2+au^2x)^2+4a^2t^2u^2x}+x+au^2x-at^2}{2}}
\end{align}
A straightforward computation shows that the maximum of RHS of \eqref{reza} is at $t=o(\sqrt{x})$. So
\begin{align*}
Re(\sqrt{x-az^2-bx-d})\leq \sqrt{x(1+au^2)}.
\end{align*}
As $c>0$, we conclude in both cases that
$e^{c\sqrt{x-az^2-bz-d}}\ll e^{c\sqrt{x(1+au^2)}}.$
Also we have $|\sin(\pi z)|\sim \frac{1}{2}e^{\pi u\sqrt{x}}$.
So we will get that for $z\in \gamma_1,\gamma_3$
\begin{align}
|H(z)|\ll e^{c\sqrt{x(1+au^2)}-\pi u\sqrt{x}}
\end{align}
We desire to make the contribution from $z\in \gamma_1,\gamma_3$ to be approximately equal to the contribution from $z\in \gamma_2,\gamma_4$. It means that we need $c\sqrt{x(1+au^2)}-\pi u\sqrt{x}<wc\sqrt{x}$, where $w$ is defined in the theorem. We need to express $u$ in terms of $w$. 
After solving this we get two cases. If $\pi^2\neq ac^2$, then
\begin{align}\label{connd1}
\frac{-cw\pi+c\sqrt{\pi^2-ac^2+ac^2w^2}}{\pi^2-ac^2}<u.
\end{align}
Otherwise, we will get
\begin{align}\label{connd2}
\frac{c(1-w^2)}{2\pi w}<u.
\end{align}

Now we compute the case $z\in \gamma_2,\gamma_4$. We have $z=\pm\sqrt{\frac{x}{a}}\mp \frac{2b}{a}+it$ and $-u\sqrt{x}<t<u\sqrt{x}$. If $t=o(\sqrt{x})$, then $\sqrt{x-q(z)}=o(\sqrt{x})$. Otherwise, using lemma \ref{l:sqrt}
\begin{align*}
Re(\sqrt{x-q(z)})\ll
\sqrt{t\sqrt{a}\frac{\sqrt{at^2+4x}+t\sqrt{a}}{2}}.
\end{align*}
Let $t=r\sqrt{x}$. We need to choose a proper $\alpha$ as follows.
\begin{align*}
\alpha=\text{argmax}_{r} \left(c\sqrt{r\sqrt{a}\frac{\sqrt{ar^2+4}+r\sqrt{a}}{2}}-\pi r\right) \quad ,\quad 0\leq r\leq u.
\end{align*}
Also we assume that $\pm\sqrt{\frac{x}{a}}\mp \frac{2b}{a}$ is far enough from integers (otherwise we shift the legs $\gamma_2,\gamma_4$ slightly to avoid $Re(z)$ being near to integer). So we conclude that $|\sin(\pi rz)|>\lambda >0$ for a fixed $\lambda$. Then we have
\begin{align*}
\int_{\gamma_2,\gamma_4}\frac{e^{c\sqrt{x-q(z)}}h(z)}{\sin(\pi z)}\ll \sqrt{x}e^{c\sqrt{x\alpha\sqrt{a}\frac{\sqrt{a\alpha^2+4}+\alpha\sqrt{a}}{2}}-\pi \alpha\sqrt{x}} h(\sqrt{x})
\end{align*}
Finally in order to satisfy \eqref{connd1} and \eqref{connd2} and the fact that $u\geq r$, we choose
\begin{align*}
u=\max\left(\frac{-cw\pi+c\sqrt{\pi^2-ac^2+ac^2w^2}}{\pi^2-ac^2}, \alpha\right)\quad \text{or}\quad u=\max\left(\frac{c(1-w^2)}{2\pi w}
,\alpha\right).
\end{align*}
where
\begin{align*}
w=\sqrt{\alpha\sqrt{a}\frac{\sqrt{a\alpha^2+4}+\alpha\sqrt{a}}{2}}-\frac{\pi \alpha}{c}.
\end{align*}
\end{proof}


\bigskip

In this paper, we need two versions of the Van der Corput lemma.
The versions we give here are a little different than what is known in \cite{tao}. But these versions are straightforward and enough for the purpose of this paper.


\begin{lemma}\label{vandercorput}
[Simpler version] Let $F(x)$ be a second differentiable function in $(a,b)$; also $0<M<|F'(x)|$, and $
|G(x)|<R$ for $x\in (a,b)$. Assume that $\frac{G(x)}{F'(x)}$ is a piecewise monotone function. Then
\begin{align}
\int_{a}^{b}e^{iF(x)}G(x)dx\ll \frac{R}{M}.
\end{align}
\end{lemma}



\begin{lemma}\cite{montgomerybook}
\label{vandercorputprocessb} 
Suppose that $f(x)$ is a real-valued function such that $0<\lambda_2\leq f''(x)$ for all $x\in [a,b],$ and suppose that $|f^{(3)}(x)|\leq \lambda_3$ and that $|f^{(4)}(x)|\leq \lambda_4$ throughout this interval. Put $f'(a)=\alpha$, $f'(b)=\theta$. For integers $\nu\in [\alpha-1,\theta+1]$ let $x_{\nu}$ be the root of the equation $f'(x)=\nu$. Then
\begin{align*}
\sum_{a\leq n\leq b}e^{2\pi if(n)}=e^{\frac{\pi i}{4}}\sum_{\alpha-1\leq \nu\leq \theta+1}\frac{e^{2\pi i(f(x_{\nu})-\nu x_{\nu})}}{\sqrt{f''(x_{\nu})}}&+O\left(\log(4+\theta-\alpha)\right) +O(\lambda_2^{-\frac{1}{2}}(\theta-\alpha+2))
\nonumber\\
&+O\left((\lambda_3^2\lambda_2^{-3}+\lambda_4\lambda_2^{-2})(b-a)(\theta-\alpha+2)\right).
\end{align*}
\end{lemma}

Note that if $f''(x)<-\lambda_2<0$, then $e^{\frac{\pi i}{4}}$ will change $e^{-\frac{\pi i}{4}}$.

\bigskip


\textit{Proof of Theorem \ref{main2}. }

Let $T>0$ and $\gamma$ be the contour with vertices
\begin{align*}
\pm \sqrt{\eta xT}\pm iu\sqrt{x}
\end{align*}
where $\eta=\frac{\pi^2}{1+\pi^2}-\epsilon$ and  $0<u$ will be determined later (see figure \ref{shekl2}).
Let
\begin{align*}
h_{T}(z)=\frac{e^{(\alpha+i\beta)\sqrt{x-\frac{z^2}{T}}}}{\sin(\pi z)}
\end{align*}


\begin{figure}[h!]

\tikzset{every picture/.style={line width=0.75pt}} 

\begin{tikzpicture}[x=0.75pt,y=0.75pt,yscale=-1,xscale=1]

\draw [line width=1.5] (301.98,1497.2) -- (538.2,1497.2)(325.6,1377.5) -- (325.6,1510.5) (531.2,1492.2) -- (538.2,1497.2) -- (531.2,1502.2) (320.6,1384.5) -- (325.6,1377.5) -- (330.6,1384.5) ;
\draw [line width=1.5] (125,1495.2) -- (302,1497.2) ;
\draw [line width=2.25] (325,1507.2) -- (326.3,1584.68) -- (327,1594.2) ;
\draw (161.1,1430.7) -- (490.1,1430.7) -- (490.1,1562.7) -- (161.1,1562.7) -- cycle ;
\draw [color={rgb, 255:red, 208; green, 2; blue, 27 } ,draw opacity=1 ][line width=1.5] (117,1495.2) -- (154,1495.2) ;
\draw [color={rgb, 255:red, 208; green, 2; blue, 27 } ,draw opacity=1 ][line width=3] [line join = round][line cap = round] (237,1494.2) .. controls (237,1494.54) and (237,1494.87) .. (237,1495.2) ;
\draw [color={rgb, 255:red, 208; green, 2; blue, 27 } ,draw opacity=1 ][line width=3] [line join = round][line cap = round] (244,1494.2) .. controls (244,1494.2) and (244,1494.2) .. (244,1494.2) ;
\draw [color={rgb, 255:red, 208; green, 2; blue, 27 } ,draw opacity=1 ][line width=3] [line join = round][line cap = round] (251,1494.2) .. controls (251,1494.2) and (251,1494.2) .. (251,1494.2) ;
\draw [color={rgb, 255:red, 208; green, 2; blue, 27 } ,draw opacity=1 ][line width=3] [line join = round][line cap = round] (393,1495.2) .. controls (393,1495.2) and (393,1495.2) .. (393,1495.2) ;
\draw [color={rgb, 255:red, 208; green, 2; blue, 27 } ,draw opacity=1 ][line width=3] [line join = round][line cap = round] (401,1495.2) .. controls (401,1495.54) and (401,1495.87) .. (401,1496.2) ;
\draw [color={rgb, 255:red, 208; green, 2; blue, 27 } ,draw opacity=1 ][line width=3] [line join = round][line cap = round] (409,1495.2) .. controls (409,1495.54) and (409,1495.87) .. (409,1496.2) ;
\draw (311,1562.2) -- (445,1563.19) ;
\draw [shift={(447,1563.2)}, rotate = 180.42] [color={rgb, 255:red, 0; green, 0; blue, 0 } ][line width=0.75] (10.93,-3.29) .. controls (6.95,-1.4) and (3.31,-0.3) .. (0,0) .. controls (3.31,0.3) and (6.95,1.4) .. (10.93,3.29) ;
\draw (524,1526.2) -- (496.54,1503.48) ;
\draw [shift={(495,1502.2)}, rotate = 399.61] [color={rgb, 255:red, 0; green, 0; blue, 0 } ][line width=0.75] (10.93,-3.29) .. controls (6.95,-1.4) and (3.31,-0.3) .. (0,0) .. controls (3.31,0.3) and (6.95,1.4) .. (10.93,3.29) ;
\draw (134,1460.2) -- (155.76,1487.64) ;
\draw [shift={(157,1489.2)}, rotate = 231.57999999999998] [color={rgb, 255:red, 0; green, 0; blue, 0 } ][line width=0.75] (10.93,-3.29) .. controls (6.95,-1.4) and (3.31,-0.3) .. (0,0) .. controls (3.31,0.3) and (6.95,1.4) .. (10.93,3.29) ;
\draw (369,1431.7) -- (234,1430.72) ;
\draw [shift={(232,1430.7)}, rotate = 360.41999999999996] [color={rgb, 255:red, 0; green, 0; blue, 0 } ][line width=0.75] (10.93,-3.29) .. controls (6.95,-1.4) and (3.31,-0.3) .. (0,0) .. controls (3.31,0.3) and (6.95,1.4) .. (10.93,3.29) ;
\draw (172,1370.7) ;
\draw [color={rgb, 255:red, 208; green, 2; blue, 27 } ,draw opacity=1 ] (477,1497.7) ;
\draw [shift={(477,1497.7)}, rotate = 0] [color={rgb, 255:red, 208; green, 2; blue, 27 } ,draw opacity=1 ][fill={rgb, 255:red, 208; green, 2; blue, 27 } ,fill opacity=1 ][line width=0.75] (0, 0) circle [x radius= 3.35, y radius= 3.35] ;
\draw [color={rgb, 255:red, 208; green, 2; blue, 27 } ,draw opacity=1 ][line width=1.5] (506,1496.7) -- (533,1496.7) ;
\draw [color={rgb, 255:red, 208; green, 2; blue, 27 } ,draw opacity=1 ] (457,1497.7) ;
\draw [shift={(457,1497.7)}, rotate = 0] [color={rgb, 255:red, 208; green, 2; blue, 27 } ,draw opacity=1 ][fill={rgb, 255:red, 208; green, 2; blue, 27 } ,fill opacity=1 ][line width=0.75] (0, 0) circle [x radius= 3.35, y radius= 3.35] ;
\draw [color={rgb, 255:red, 208; green, 2; blue, 27 } ,draw opacity=1 ] (435,1497.7) ;
\draw [shift={(435,1497.7)}, rotate = 0] [color={rgb, 255:red, 208; green, 2; blue, 27 } ,draw opacity=1 ][fill={rgb, 255:red, 208; green, 2; blue, 27 } ,fill opacity=1 ][line width=0.75] (0, 0) circle [x radius= 3.35, y radius= 3.35] ;
\draw [color={rgb, 255:red, 208; green, 2; blue, 27 } ,draw opacity=1 ] (373,1496.7) ;
\draw [shift={(373,1496.7)}, rotate = 0] [color={rgb, 255:red, 208; green, 2; blue, 27 } ,draw opacity=1 ][fill={rgb, 255:red, 208; green, 2; blue, 27 } ,fill opacity=1 ][line width=0.75] (0, 0) circle [x radius= 3.35, y radius= 3.35] ;
\draw [color={rgb, 255:red, 208; green, 2; blue, 27 } ,draw opacity=1 ] (350,1496.7) ;
\draw [shift={(350,1496.7)}, rotate = 0] [color={rgb, 255:red, 208; green, 2; blue, 27 } ,draw opacity=1 ][fill={rgb, 255:red, 208; green, 2; blue, 27 } ,fill opacity=1 ][line width=0.75] (0, 0) circle [x radius= 3.35, y radius= 3.35] ;
\draw [color={rgb, 255:red, 208; green, 2; blue, 27 } ,draw opacity=1 ] (325.6,1497.2) ;
\draw [shift={(325.6,1497.2)}, rotate = 0] [color={rgb, 255:red, 208; green, 2; blue, 27 } ,draw opacity=1 ][fill={rgb, 255:red, 208; green, 2; blue, 27 } ,fill opacity=1 ][line width=0.75] (0, 0) circle [x radius= 3.35, y radius= 3.35] ;
\draw [color={rgb, 255:red, 208; green, 2; blue, 27 } ,draw opacity=1 ] (302,1497.2) ;
\draw [shift={(302,1497.2)}, rotate = 0] [color={rgb, 255:red, 208; green, 2; blue, 27 } ,draw opacity=1 ][fill={rgb, 255:red, 208; green, 2; blue, 27 } ,fill opacity=1 ][line width=0.75] (0, 0) circle [x radius= 3.35, y radius= 3.35] ;
\draw [color={rgb, 255:red, 208; green, 2; blue, 27 } ,draw opacity=1 ] (282,1496.7) ;
\draw [shift={(282,1496.7)}, rotate = 0] [color={rgb, 255:red, 208; green, 2; blue, 27 } ,draw opacity=1 ][fill={rgb, 255:red, 208; green, 2; blue, 27 } ,fill opacity=1 ][line width=0.75] (0, 0) circle [x radius= 3.35, y radius= 3.35] ;
\draw [color={rgb, 255:red, 208; green, 2; blue, 27 } ,draw opacity=1 ] (213.5,1496.2) ;
\draw [shift={(213.5,1496.2)}, rotate = 0] [color={rgb, 255:red, 208; green, 2; blue, 27 } ,draw opacity=1 ][fill={rgb, 255:red, 208; green, 2; blue, 27 } ,fill opacity=1 ][line width=0.75] (0, 0) circle [x radius= 3.35, y radius= 3.35] ;
\draw [color={rgb, 255:red, 208; green, 2; blue, 27 } ,draw opacity=1 ] (194,1495.7) ;
\draw [shift={(194,1495.7)}, rotate = 0] [color={rgb, 255:red, 208; green, 2; blue, 27 } ,draw opacity=1 ][fill={rgb, 255:red, 208; green, 2; blue, 27 } ,fill opacity=1 ][line width=0.75] (0, 0) circle [x radius= 3.35, y radius= 3.35] ;
\draw [color={rgb, 255:red, 208; green, 2; blue, 27 } ,draw opacity=1 ] (173,1495.7) ;
\draw [shift={(173,1495.7)}, rotate = 0] [color={rgb, 255:red, 208; green, 2; blue, 27 } ,draw opacity=1 ][fill={rgb, 255:red, 208; green, 2; blue, 27 } ,fill opacity=1 ][line width=0.75] (0, 0) circle [x radius= 3.35, y radius= 3.35] ;

\draw (77,1434.2) node [anchor=north west][inner sep=0.75pt] [xslant=-0.04] [align=left] {\mbox{-}$\displaystyle \sqrt{Tx\eta }$};
\draw (517,1515.2) node [anchor=north west][inner sep=0.75pt] [xslant=-0.04] [align=left] {$\displaystyle \sqrt{Tx\eta }$};
\draw (327,1559.2) node [anchor=north west][inner sep=0.75pt] [align=left] {$\displaystyle -u\sqrt{x}$};
\draw (327,1398.2) node [anchor=north west][inner sep=0.75pt] [align=left] {$\displaystyle u\sqrt{x}$};
\draw (489.75,1430.4) node [anchor=north west][inner sep=0.75pt] [align=left] {$\displaystyle \gamma _{4}$};
\draw (442,1561.2) node [anchor=north west][inner sep=0.75pt] [align=left] {$\displaystyle \gamma _{3}$};
\draw (138,1531.2) node [anchor=north west][inner sep=0.75pt] [align=left] {$\displaystyle \gamma _{2}$};
\draw (208,1401.2) node [anchor=north west][inner sep=0.75pt] [align=left] {$\displaystyle \gamma _{1}$};

\end{tikzpicture}

\centering \caption{Contour $\gamma$ for complex $c$ case}
\label{shekl2}
\end{figure}
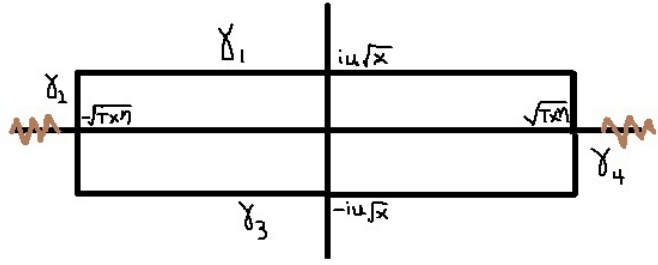


We take the branch cut to be $(-\infty,-\sqrt{xT}]\cup [\sqrt{xT},\infty)$. The Residue Theorem implies
\begin{align}
\int_{\gamma}h_{T}(z)dz=2\pi i\sum_{\ell^2<Tx\eta}(-1)^\ell e^{(\alpha+i\beta)\sqrt{x-\frac{\ell^2}{T}}}.
\end{align}
Now we compute the case where $z\in \gamma_1,\gamma_3$. So $z=t\pm iu\sqrt{x}$ and $-\sqrt{Tx\eta}<t<\sqrt{Tx\eta}$. As $x\ll \sqrt{T}$ we have
\begin{align*}
\sqrt{x-\frac{z^2}{T}}
=\sqrt{x-\frac{t^2-u^2x\pm 2iut\sqrt{x}}{T}}\sim \sqrt{x+\frac{-t^2\mp 2iut\sqrt{x}}{T}+O(\frac{1}{\sqrt{T}})}
\end{align*}
If $t\sqrt{x}=o(\sqrt{T})$, then noting $x-\frac{t^2}{T}=x+o(\frac{1}{x})$ we conclude that
\begin{align}
\sqrt{x-\frac{z^2}{T}}\sim\sqrt{x} + C(x,T)+iB(x,T), \text{where} B(x,T)=o(\frac{1}{\sqrt{T}}) \text{ and }C(x,T)=o(\frac{1}{x}).
\end{align}
Otherwise, recall that in the worst case, $t\sim \sqrt{xT}$. We use lemma \ref{l:sqrt} to get
\begin{align}\label{ghghabli}
\sqrt{x-\frac{z^2}{T}}
&\sim \sqrt{\frac{\sqrt{(x-\frac{t^2}{T})^2+\frac{4u^2t^2x}{T^2}}+(x-\frac{t^2}{T})}{2}}+i\sqrt{\frac{\sqrt{(x-\frac{t^2}{T})^2+\frac{4u^2t^2x}{T^2}}-(x-\frac{t^2}{T})}{2}}
\end{align}
Because of the range of values of $t$, $x-\frac{t^2}{T}\geq (1-\eta)x$.
So
\begin{align*}
\sqrt{(x-\frac{t^2}{T})^2+\frac{4u^2t^2x}{T^2}}-(x-\frac{t^2}{T})= \frac{\frac{4u^2t^2x}{T^2}}{\sqrt{(x-\frac{t^2}{T})^2+\frac{4u^2t^2x}{T^2}}+(x-\frac{t^2}{T})}\leq
\frac{2u^2t^2}{(1-\eta)T^2}.
\end{align*}
Therefore in all cases for $t$
\begin{align*}\label{reim}
Re(\sqrt{x-\frac{z^2}{T}})&\leq \sqrt{x}\quad \text{and} \quad
Im(\sqrt{x-\frac{z^2}{T}})
\leq \sqrt{\frac{xu^2\eta}{(1-\eta)T}}.
\end{align*}
Noting that $|\beta|\leq \sqrt{T}$
\begin{align*}
Re\left((\alpha+i\beta)\sqrt{x-\frac{z^2}{T}}\right)
\leq \alpha\sqrt{x}+\sqrt{\frac{xu^2\eta}{(1-\eta)}}\frac{|\beta|}{\sqrt{T}}\leq \sqrt{x}\left(\alpha+u\sqrt{\frac{\eta}{(1-\eta)}}\right)
\end{align*}
So for $z\in \gamma_1,\gamma_3$ we have
\begin{align}
\frac{e^{(\alpha+i\beta)\sqrt{x-\frac{z^2}{T}}}}{\sin(\pi z)}\ll e^{\left(\alpha+u\sqrt{\frac{\eta}{(1-\eta)}}-\pi u\right)\sqrt{x}}.
\end{align}
We will later choose proper $w,\eta,u$ such that
\begin{align}\label{thecondition}
\alpha+u\sqrt{\frac{\eta}{(1-\eta)}}-\pi u + \frac{\log(T)}{2\sqrt{x}}<w.
\end{align}
Next we assume that $z\in \gamma_2,\gamma_4$.
We have $z=\pm \sqrt{\eta xT}+it$ and $-u\sqrt{x}\leq t\leq u\sqrt{x}$.
As $t\ll \sqrt{x}\ll \sqrt[4]{T}$, then $\frac{t\sqrt{x}}{\sqrt{T}}\ll 1$ and $\frac{t^2}{T}\ll \frac{x}{T}=o(1)$. We use lemma \ref{l:sqrt} to conclude that
\begin{align*}
\sqrt{x-\frac{z^2}{T}}
&\sim \sqrt[4]{\frac{x}{T}}\left(\sqrt{\frac{\sqrt{4t^2\eta+xT(1-\eta)^2}+(1-\eta)\sqrt{xT}}{2}}+i\sqrt{\frac{\sqrt{4t^2\eta+xT(1-\eta)^2}-(1-\eta)\sqrt{xT}}{2}}\right)\nonumber\\
&\quad\quad\quad\quad\quad\quad\quad\quad\quad\quad\quad\quad\quad\quad\quad\quad\quad\quad\quad\quad\sim \sqrt{(1-\eta)x}+it\sqrt{\frac{\eta}{T(1-\eta)}}.
\end{align*}
This together with the fact that $|\beta|<\sqrt{T}$ imply that
\begin{align*}
Re\left((\alpha+i\beta)\sqrt{x-\frac{z^2}{T}}\right)
\leq \alpha\sqrt{(1-\eta)x}+\frac{|\beta|}{\sqrt{T}}\sqrt{\frac{\eta t^2}{1-\eta}}\leq \alpha\sqrt{(1-\eta)x}+|t|\sqrt{\frac{\eta}{1-\eta}}
\end{align*}
Again we assume that $\sqrt{xT\eta}$ is far away from the integers; so as $\sin(\pi z)>\lambda>0$ for some fixed $\lambda$, therefore for $z\in \gamma_2,\gamma_4$
\begin{align*}
\frac{e^{(\alpha+i\beta)\sqrt{x-\frac{z^2}{T}}}}{\sin(\pi z)}\ll e^{\alpha\sqrt{(1-\eta)x}+|t|\left(\sqrt{\frac{\eta}{1-\eta}}-\pi\right)}.
\end{align*}
As $\sqrt{\frac{\eta}{1-\eta}}<\pi$, the maximum of the following function occurs at $y=0$:
\begin{align*}
G(y)=\alpha\sqrt{(1-\eta)}+y\sqrt{\frac{\eta}{1-\eta}}-y\pi.
\end{align*}
We conclude that
\begin{align}\label{tikeaval}
2\pi i\sum_{\ell^2<Tx\eta}(-1)^\ell&e^{(\alpha+i\beta)\sqrt{x-\frac{\ell^2}{T}}}=\int_{\gamma}h_{T}(z)dz\nonumber\\
&\ll \sqrt{Tx}e^{(\alpha+u\sqrt{\frac{\eta}{(1-\eta)}}-\pi u)\sqrt{x}} + \sqrt{x}e^{\alpha\sqrt{(1-\eta)}\sqrt{x}}.
\end{align}
A straightforward calculation shows that
\begin{align*}
\left\vert\sum_{xT\eta\leq \ell^2<Tx}(-1)^le^{(\alpha+i\beta)\sqrt{x-\frac{\ell^2}{T}}}\right\vert \ll \sum_{xT\eta\leq \ell^2<Tx}e^{\alpha\sqrt{x-\frac{\ell^2}{T}}}\ll\sqrt{xT}e^{\alpha\sqrt{x(1-\eta)}}.
\end{align*}
For a sharper bound, we use lemma \ref{vandercorputprocessb} to control the tail. Without loss of generality assume that $\beta<0$. We prove that for $|\beta|\ll \sqrt{T}$, 
\begin{align}\label{e:oscilationsum}
\left\vert\sum_{xT\eta \leq \ell^2<Tx-T}(-1)^\ell e^{i\beta\sqrt{x-\frac{\ell^2}{T}}}\right\vert  \ll \frac{\sqrt{Tx^{3/2}}}{\sqrt{|\beta|+1}}+  \frac{\sqrt{Tx^{11}}}{|\beta|+1}+\log x .
\end{align}
It is trivial to get the bound for $|\beta|<1$, so we assume otherwise. Let $f(\ell):=\frac{1}{2} \ell+\frac{\beta}{2\pi}\left(x-\frac{\ell^2}{T}\right)^{1/2}$. Then 
\begin{align*}
    f'(\ell) &= \frac{1}{2} -\frac{\beta \ell}{2\pi T}\left(x-\frac{\ell^2}{T}\right)^{-1/2}
    && f''(\ell) = -\frac{\beta x}{2\pi T}(x-\frac{\ell^2}{T})^{-3/2}\nonumber\\
    f^{(3)}(\ell) &= -\frac{3\beta x\ell}{2\pi T^2}(x-\frac{\ell^2}{T})^{-5/2}
    && f^{(4)}(\ell) = -\frac{3\beta x}{2\pi T^2}(x+\frac{4\ell^2}{T})(x-\frac{\ell^2}{T})^{-7/2}.
\end{align*}
First we evaluate $f'$ at the endpoints. Without loss of generality we consider $\sqrt{xT\eta}<\ell<\sqrt{xT-T}$. Assuming that $\eta<0.9$, we have
\begin{align}
    f'(\sqrt{xT\eta}) &= \frac{1}{2} - \frac{\beta\sqrt{x\eta}}{2\pi \sqrt{T}}(x-\eta x)^{-1/2} \in [0,1]\\
    f'(\sqrt{xT-T}) &= \frac{1}{2} - \frac{\beta\sqrt{xT-T}}{2\pi T} <\sqrt{x}.
\end{align}
As $f'',f^{(3)},f^{(4)},f^{(5)}$ are positive, we can find $\lambda_2,\lambda_3,\lambda_4$ easily at the endpoints. 
\begin{align}
    \lambda_2 &= \inf_{\sqrt{Tx\eta}<\ell<\sqrt{xT-T}} f''(\ell) = f''(\sqrt{xT\eta}) = \frac{|\beta| }{2\pi  T\sqrt{x}(1-\eta)^{3/2}}\gg \frac{|\beta|}{T\sqrt{x}}\nonumber\\
    \lambda_3 &=\sup_{\sqrt{Tx\eta}<\ell<\sqrt{xT-T}} f^{(3)}(\ell) = f^{(3)}(\sqrt{xT-T}) = \frac{-3\beta x}{2\pi T^2}(xT-T)^{1/2}\ll \frac{|\beta|x^{3/2}}{T^{3/2}}\nonumber\\
    \lambda_4 &=\sup_{\sqrt{Tx\eta}<\ell<\sqrt{xT-T}} f^{(4)}(\ell) = f^{(4)}(\sqrt{xT-T}) = \frac{-3\beta x}{2\pi T^2}(5x-4)\ll \frac{|\beta|x^2}{T^2}.
\end{align}
It implies that
\begin{align*}
    \lambda_3^2\lambda_2^{-3}+\lambda_4\lambda_2^{-2} \ll \frac{x^{9/2}}{|\beta|+1}.
\end{align*}
Noting that $f'(\sqrt{xT-T})-f'(\sqrt{xT\eta}) =O(\sqrt{x})$ and applying lemma \ref{vandercorputprocessb} implies the equation \eqref{e:oscilationsum}.
For $xT\eta<t^2<xT-T$  we define
$$S(t):=\sum_{xT\eta\leq \ell^2<t^2}(-1)^\ell e^{i\beta\sqrt{x-\frac{\ell^2}{T}}}.$$ 
Similar to what we just did, we know that $|S(t)|\ll x^5t/\sqrt{|\beta|}\ll\sqrt{x^{11}T/|\beta|}$ for $xT\eta<t^2<xT-T$. Using Abel's summation formula we get
\begin{align}\label{abelsumformula}
\Bigg\vert\sum_{xT\eta\leq \ell^2<Tx-T}&(-1)^\ell e^{(\alpha+i\beta)\sqrt{x-\frac{\ell^2}{T}}}\Bigg\vert \nonumber\\
&\ll |S(\sqrt{xT-T})|e^{\alpha}+|S(\sqrt{\eta xT})|e^{\alpha\sqrt{(1-\eta)x}}+\left\vert\int_{\sqrt{\eta xT}}^{\sqrt{xT-T}}S(t)v(t)dt\right\vert
\end{align}
where $v(t):= \frac{d}{dt}\exp\left(\alpha\sqrt{x-\frac{t^2}{T}}\right)$. We bound the integral in the RHS.
\begin{align*}
\int_{\sqrt{\eta xT}}^{\sqrt{xT-T}}S(t)v(t)dt &= -\frac{\alpha}{T}\int_{\sqrt{\eta xT}}^{\sqrt{xT-T}}  \frac{tS(t) e^{\alpha\sqrt{x-\frac{t^2}{T}}}}{\sqrt{x-\frac{t^2}{T}}}dt.
\end{align*}
 Straightforward computation  gives that
\begin{align*}
G(t):= \frac{tS(t) e^{\alpha\sqrt{x-\frac{t^2}{T}}}}{\sqrt{x-\frac{t^2}{T}}} &\ll \left(\frac{\sqrt{Tx^{3/2}}}{\sqrt{|\beta|+1}}+ \log x + \frac{\sqrt{Tx^{11}}}{|\beta|+1}\right)e^{\alpha\sqrt{(1-\eta)x}}\sqrt{xT}
\end{align*}
It implies that
\begin{align*}
\left\vert\int_{\sqrt{\eta xT}}^{\sqrt{xT-T}}S(t)v(t)dt \right\vert&\ll \frac{\sqrt{xT}}{T}\left(\frac{\sqrt{Tx^{3/2}}}{\sqrt{|\beta|+1}}+ \log x + \frac{\sqrt{Tx^{11}}}{|\beta|+1}\right)e^{\alpha\sqrt{(1-\eta)x}}\sqrt{xT} \nonumber\\
&\ll  \left(\frac{\sqrt{Tx^{5/2}}}{\sqrt{|\beta|+1}}+\sqrt{x} \log x + \frac{\sqrt{Tx^{12}}}{|\beta|+1}\right)e^{\alpha\sqrt{(1-\eta)x}}.
\end{align*}
This, \eqref{e:oscilationsum}, and \eqref{abelsumformula} give
\begin{align*}
\Bigg\vert\sum_{xT\eta\leq \ell^2<Tx-T}&(-1)^\ell e^{(\alpha+i\beta)\sqrt{x-\frac{\ell^2}{T}}}\Bigg\vert \ll \left(\frac{\sqrt{Tx^{5/2}}}{\sqrt{|\beta|+1}}+\sqrt{x} \log x + \frac{\sqrt{Tx^{12}}}{|\beta|+1}\right)e^{\alpha\sqrt{(1-\eta)x}}.
\end{align*}
Considering the range of $\beta$ in our application, we conclude that
\begin{align}\label{tail}
\Bigg\vert \sum_{xT\eta\leq \ell^2<Tx}(-1)^\ell e^{(\alpha+i\beta)\sqrt{x-\frac{\ell^2}{T}}}\Bigg\vert &\leq \Bigg\vert \sum_{xT\eta\leq \ell^2<Tx-T}(-1)^\ell e^{(\alpha+i\beta)\sqrt{x-\frac{\ell^2}{T}}}\Bigg\vert+\Bigg\vert \sum_{xT-T\leq \ell^2<Tx}(-1)^\ell e^{(\alpha+i\beta)\sqrt{x-\frac{\ell^2}{T}}}\Bigg\vert \nonumber\\
&\ll \sqrt{\frac{x^3T}{|\beta|+1}} e^{\alpha\sqrt{(1-\eta)x}} + \sqrt{Tx}.
\end{align}
We used a trivial bound for the second sum. We want to have
\begin{align*}
\Bigg\vert \sum_{ \ell^2<Tx}(-1)^\ell e^{(\alpha+i\beta)\sqrt{x-\frac{\ell^2}{T}}}\Bigg\vert \ll \sqrt{\frac{T}{|\beta|+1}}e^{w\sqrt{x}}.
\end{align*}
Adding \eqref{tikeaval} and \eqref{tail} we need to have 
\begin{align}\label{cond2}
\begin{cases}
\alpha\sqrt{1-\eta}\leq w\\
\alpha+u\sqrt{\frac{\eta}{(1-\eta)}}-\pi u + \frac{1}{2\sqrt{x}}\log(|\beta|+1)\leq w.
\end{cases}
\end{align}
Remember that $\eta=\frac{\pi^2}{1+\pi^2}-\epsilon$. Comparing with $\beta$, if we choose $u$ large enough  then the left hand side of the second condition in \eqref{cond2} becomes negative. So
$$w=\alpha\sqrt{\frac{1}{1+\pi^2}}+\epsilon $$
from the first condition.
This completes the proof.
\qed


\bigskip


\section{Proof related to prime distribution}

Inspired by the proof of the Prime Number Theorem (PNT) we compute the following sum in two ways.
\begin{align}\label{equu}
\frac{1}{2\pi i}\sum_{\ell^2<Tx}(-1)^\ell\int_{1+\epsilon-i\sqrt{T}}^{1+\epsilon+i\sqrt{T}}\frac{\zeta'(s)}{\zeta(s)}\frac{e^{s\sqrt{x-\frac{\ell^2}{T}}}}{s}ds.
\end{align}
In this section, we assume that $T<e^{\frac{4}{3}\sqrt{x}}$.

\bigskip


\begin{lemma}
For large enough $x$
\begin{align}\label{firstway}
\sum_{\ell^2<Tx}&(-1)^\ell\int_{1+\epsilon-i\sqrt{T}}^{1+\epsilon+i\sqrt{T}}\frac{\zeta'(s)}{\zeta(s)}\frac{e^{s\sqrt{x-\frac{\ell^2}{T}}}}{s}ds\ll \sqrt{Tx}e^{(\frac{1}{\sqrt{1+\pi^2}}+\epsilon)\sqrt{x}}.
\end{align}
\end{lemma}



\begin{proof}
We consider the contour $\gamma$ in figure \ref{shekll3}, where $\epsilon>0$ is a very small real number and $U$ is a very large real number far enough from any negative even integer $-2m$. Using the Residue Theorem
\begin{align}\label{terms}
\sum_{\ell^2<Tx}&(-1)^\ell\int_{\gamma}\frac{\zeta'(s)}{\zeta(s)}\frac{e^{s\sqrt{x-\frac{\ell^2}{T}}}}{s}ds\nonumber\\
&=2\pi i\sum_{\ell^2<Tx}(-1)^\ell\Bigg(\Big(\lim_{s\rightarrow 0}e^{s\sqrt{x-\frac{\ell^2}{T}}}\big(\frac{\zeta'(s)}{\zeta(s)}\big)\Big)+e^{\sqrt{x-\frac{\ell^2}{T}}}\nonumber\\
&\quad\quad\quad\quad\quad\quad\quad\quad\quad+\sum_{|Im(\rho_m)|<\sqrt{T}}\frac{e^{\rho_m\sqrt{x-\frac{\ell^2}{T}}}}{\rho_m}-\sum_{1\leq m\leq U/2}\frac{e^{-2m\sqrt{x-\frac{\ell^2}{T}}}}{2m}\Bigg),
\end{align}
where $\rho_m$ is the $m^{th}$ non trivial zeroes of the Rieman zeta function.


\begin{figure}[h!] 

\tikzset{every picture/.style={line width=0.75pt}} 

\begin{tikzpicture}[x=0.75pt,y=0.75pt,yscale=-1,xscale=1]

\draw [line width=1.5]  (311,1839.12) -- (382,1839.12)(318.1,1718.7) -- (318.1,1852.5) (375,1834.12) -- (382,1839.12) -- (375,1844.12) (313.1,1725.7) -- (318.1,1718.7) -- (323.1,1725.7)  ;
\draw [line width=1.5]    (167,1837.7) -- (312,1838.7) ;
\draw [line width=1.5]    (318,1841.7) -- (319,1961.7) ;
\draw  [color={rgb, 255:red, 65; green, 117; blue, 5 }  ,draw opacity=1 ][line width=1.5]  (177,1743.7) -- (367,1743.7) -- (367,1933.7) -- (177,1933.7) -- cycle ;
\draw [color={rgb, 255:red, 208; green, 2; blue, 27 }  ,draw opacity=1 ] [dash pattern={on 4.5pt off 4.5pt}]  (336,1730.7) -- (337,1949.7) ;
\draw    (367,1724.7) -- (340.86,1734.97) ;
\draw [shift={(339,1735.7)}, rotate = 338.55] [color={rgb, 255:red, 0; green, 0; blue, 0 }  ][line width=0.75]    (10.93,-3.29) .. controls (6.95,-1.4) and (3.31,-0.3) .. (0,0) .. controls (3.31,0.3) and (6.95,1.4) .. (10.93,3.29)   ;
\draw    (386,1820.7) -- (372.46,1833.34) ;
\draw [shift={(371,1834.7)}, rotate = 316.97] [color={rgb, 255:red, 0; green, 0; blue, 0 }  ][line width=0.75]    (10.93,-3.29) .. controls (6.95,-1.4) and (3.31,-0.3) .. (0,0) .. controls (3.31,0.3) and (6.95,1.4) .. (10.93,3.29)   ;
\draw    (300,1722.7) -- (311.83,1739.08) ;
\draw [shift={(313,1740.7)}, rotate = 234.16] [color={rgb, 255:red, 0; green, 0; blue, 0 }  ][line width=0.75]    (10.93,-3.29) .. controls (6.95,-1.4) and (3.31,-0.3) .. (0,0) .. controls (3.31,0.3) and (6.95,1.4) .. (10.93,3.29)   ;
\draw    (297,1942.7) -- (313.07,1938.24) ;
\draw [shift={(315,1937.7)}, rotate = 524.48] [color={rgb, 255:red, 0; green, 0; blue, 0 }  ][line width=0.75]    (10.93,-3.29) .. controls (6.95,-1.4) and (3.31,-0.3) .. (0,0) .. controls (3.31,0.3) and (6.95,1.4) .. (10.93,3.29)   ;

\draw (178,1816.7) node [anchor=north west][inner sep=0.75pt]  [font=\small] [align=left] {{\small -U}};
\draw (261,1932.7) node [anchor=north west][inner sep=0.75pt]  [font=\footnotesize] [align=left] {$\displaystyle -\sqrt{T}$};
\draw (273,1703.7) node [anchor=north west][inner sep=0.75pt]  [font=\footnotesize] [align=left] {$\displaystyle \sqrt{T}$};
\draw (379,1800.7) node [anchor=north west][inner sep=0.75pt]   [align=left] {1+$\displaystyle \epsilon $};
\draw (321,1821.7) node [anchor=north west][inner sep=0.75pt]  [font=\footnotesize] [align=left] {$\displaystyle 0.5$};
\draw (354,1702.7) node [anchor=north west][inner sep=0.75pt]   [align=left] {Critical line};

\end{tikzpicture}

\centering \caption{The contour $\gamma$} 
\label{shekll3}
\end{figure}
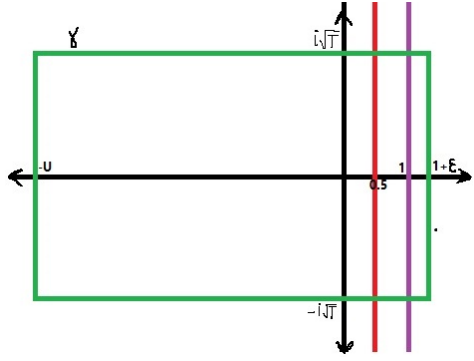


An easy computation shows that the first and fourth terms in the RHS sum have contribution at most $\sqrt{Tx}$. Using Theorem \ref{main1} (by tending $c\rightarrow 0$) the second term is bounded from above  by $\sqrt{Tx}e^{\epsilon\sqrt{x}}$ .
We can use Theorem \ref{main2} to show that 
$$\sum_{\ell^2<Tx}(-1)^\ell e^{\rho_m\sqrt{x-\frac{\ell^2}{T}}}\ll \sqrt{\frac{Tx}{|\text{Im}(\rho_m)|+1}}e^{(\frac{1}{\sqrt{1+\pi^2}}+\epsilon)\sqrt{x}}.$$ Finally, using the fact that 
$$\sum_{|\text{Im}(\rho_m)|<\sqrt{T}}\frac{1}{\text{Im}(\rho_m)^{\frac{3}{2}}}$$ 
converges, we can conclude that the third term of RHS of \eqref{terms} has contribution at most $\sqrt{Tx}e^{(\frac{1}{\sqrt{1+\pi^2}}+\epsilon)\sqrt{x}}$. So we have
\begin{align}\label{residueee}
\sum_{\ell^2<Tx}&(-1)^\ell\int_{\gamma}\frac{\zeta'(s)}{\zeta(s)}\frac{e^{s\sqrt{x-\frac{\ell^2}{T}}}}{s}ds\ll \sqrt{Tx}e^{(\frac{1}{\sqrt{1+\pi^2}}+\epsilon)\sqrt{x}}.
\end{align}
As $\frac{e^{s\sqrt{x-\frac{\ell^2}{T}}}}{s}$ tends to zero for Re$(s)\rightarrow -\infty$, we can pick $U$ large enough to have 
\begin{align}\label{taghsim}
\sum_{\ell^2<Tx}(-1)^\ell\int_{\gamma}\frac{\zeta'(s)}{\zeta(s)}\frac{e^{s\sqrt{x-\frac{\ell^2}{T}}}}{s}ds \simeq &\sum_{\ell^2<Tx}(-1)^\ell\int_{1+\epsilon-i\sqrt{T}}^{1+\epsilon+i\sqrt{T}}\frac{\zeta'(s)}{\zeta(s)}\frac{e^{s\sqrt{x-\frac{\ell^2}{T}}}}{s}ds\nonumber\\
&+\sum_{\ell^2<Tx}(-1)^\ell\int_{-U-i\sqrt{T}}^{1+\epsilon-i\sqrt{T}}\frac{\zeta'(s)}{\zeta(s)}\frac{e^{s\sqrt{x-\frac{\ell^2}{T}}}}{s}ds\nonumber\\
&+\sum_{\ell^2<Tx}(-1)^\ell\int_{1+\epsilon+i\sqrt{T}}^{-U+i\sqrt{T}}\frac{\zeta'(s)}{\zeta(s)}\frac{e^{s\sqrt{x-\frac{\ell^2}{T}}}}{s}ds.
\end{align}
The second integral in the RHS is
\begin{align*}
\Bigg\vert\sum_{\ell^2<Tx} &(-1)^\ell \int_{-U-i\sqrt{T}}^{1+\epsilon-i\sqrt{T}}\frac{\zeta'(s)}{\zeta(s)}\frac{e^{s\sqrt{x-\frac{\ell^2}{T}}}}{s}ds\Bigg\vert\nonumber\\
&=\left\vert\sum_{\ell^2<Tx}(-1)^\ell\int_{-U-i\sqrt{T}}^{-i\sqrt{T}}\frac{\zeta'(s)}{\zeta(s)}\frac{e^{s\sqrt{x-\frac{\ell^2}{T}}}}{s}ds+\int_{-i\sqrt{T}}^{1+\epsilon-i\sqrt{T}}\frac{\zeta'(s)}{\zeta(s)}\frac{\sum_{\ell^2<Tx}(-1)^\ell e^{s\sqrt{x-\frac{\ell^2}{T}}}}{s}ds\right\vert\nonumber\\
&\ll \sqrt{Tx}\int_{-\infty}^{0}\left\vert\frac{\zeta'(\sigma-i\sqrt{T})}{\zeta(\sigma-i\sqrt{T})}\right\vert\frac{e^{\sigma\sqrt{x}}}{\sqrt{T}}d\sigma+\frac{1}{\sqrt{T}}  \int_{0}^{1+\epsilon}\left\vert\frac{\zeta'(s)}{\zeta(s)}\right\vert \sqrt[4]{Tx^2}e^{\sigma(\frac{1}{\sqrt{1+\pi^2}}+\epsilon)\sqrt{x}}d\sigma.
\end{align*}
Note that in the last inequality we used Theorem \ref{main2}. We can use the fact that $\frac{\zeta'}{\zeta}(\sigma+it)=\sum_{\rho}\frac{1}{\sigma+it-\rho}+O(\log(t))$ to choose a proper $T$ such that $\frac{\zeta'}{\zeta}(\sigma\pm i\sqrt{T})\ll \log^2(T)$ for $-\infty\leq \sigma<1+\epsilon$. So we have
\begin{align*}  
\sum_{\ell^2<Tx} &(-1)^\ell\int_{-U-i\sqrt{T}}^{1+\epsilon-i\sqrt{T}}\frac{\zeta'(s)}{\zeta(s)}\frac{e^{s\sqrt{x-\frac{\ell^2}{T}}}}{s}ds\nonumber\\
&\ll \sqrt{x}\log^2(T)+\frac{\sqrt{x}}{\sqrt[4]{T}}\int_{0}^{1+\epsilon}\left\vert\frac{\zeta'(\sigma-i\sqrt{T})}{\zeta(\sigma-i\sqrt{T})}\right\vert e^{\sigma(\frac{1}{\sqrt{1+\pi^2}}+\epsilon)\sqrt{x}}d\sigma
\nonumber\\
&\ll \frac{\log^2(T)}{\sqrt[4]{T}}e^{(1+\epsilon)\sqrt{\frac{x}{1+\pi^2}}}.
\end{align*}
The third integral can be similarly bounded. This, \eqref{residueee}, and \eqref{taghsim} give the result.
\end{proof}


\bigskip


\textit{Proof of Theorem \ref{psiresult1}. }
We compute \eqref{equu} another way. We have
\begin{align}\label{secondway}
\sum_{\ell^2<Tx}(-1)^\ell &\int_{1+\epsilon-i\sqrt{T}}^{1+\epsilon+i\sqrt{T}}\frac{\zeta'(s)}{\zeta(s)}\frac{e^{s\sqrt{x-\frac{\ell^2}{T}}}}{s}ds=\sum_{\ell^2<Tx}(-1)^\ell\int_{1+\epsilon-i\sqrt{T}}^{1+\epsilon+i\sqrt{T}}\sum_{n=1}^{\infty}\frac{\Lambda(n)}{n^s}\frac{e^{s\sqrt{x-\frac{\ell^2}{T}}}}{s}ds\nonumber\\
=&\sum_{\ell^2<Tx}(-1)^\ell\left(\sum_{1\leq n\leq e^{\sqrt{x-\frac{\ell^2}{T}}}}+\sum_{e^{\sqrt{x-\frac{\ell^2}{T}}}\leq n}\int_{1+\epsilon-i\sqrt{T}}^{1+\epsilon+i\sqrt{T}}\frac{\Lambda(n)}{n^s}\frac{e^{s\sqrt{x-\frac{\ell^2}{T}}}}{s}ds\right):=A_1+A_2
\end{align}
First we compute $A_1$.
Again, we use the contour $\gamma$ in figure \ref{shekll3} to compute the integral. Knowing $\left\vert\frac{e^{\sqrt{x-\frac{l^2}{T}}}}{n}\right\vert>1$, we conclude that the integrand is tending to zero as Re$(s)\rightarrow -\infty$. Considering sufficiently large  $U$ and Using the Residue Theorem give 
\begin{align}\label{moaadelle}
A_1\simeq &\sum_{\ell^2<Tx}(-1)^\ell\sum_{1\leq n\leq e^{\sqrt{x-\frac{\ell^2}{T}}}}\Lambda(n)\int_{\gamma}\left(\frac{e^{\sqrt{x-\frac{\ell^2}{T}}}}{n}\right)^s\frac{ds}{s}\nonumber\\
&+\sum_{\ell^2<Tx}(-1)^\ell\sum_{1\leq n\leq e^{\sqrt{x-\frac{\ell^2}{T}}}}\Lambda(n)\int_{-U+i\sqrt{T}}^{1+\epsilon+i\sqrt{T}}\left(\frac{e^{\sqrt{x-\frac{\ell^2}{T}}}}{n}\right)^s\frac{ds}{s}\nonumber\\
&-\sum_{\ell^2<Tx}(-1)^\ell\sum_{1\leq n\leq e^{\sqrt{x-\frac{\ell^2}{T}}}}\Lambda(n)\int_{-U-i\sqrt{T}}^{1+\epsilon-i\sqrt{T}}\left(\frac{e^{\sqrt{x-\frac{\ell^2}{T}}}}{n}\right)^s\frac{ds}{s}\nonumber\\
=&2\pi i\sum_{\ell^2<Tx}(-1)^\ell \Psi\left(e^{\sqrt{x-\frac{l^2}{T}}}\right)
+ \sum_{\ell^2<Tx}(-1)^\ell \sum_{n\leq e^{\sqrt{x-\frac{\ell^2}{T}}}}\Lambda(n)\int_{-U+ i\sqrt{T}}^{1+\epsilon+ i\sqrt{T}}\left(\frac{e^{\sqrt{x-\frac{\ell^2}{T}}}}{n}\right)^s\frac{ds}{s}\nonumber\\
&\quad\quad\quad\quad\quad\quad\quad -\sum_{\ell^2<Tx}(-1)^\ell  \sum_{n\leq e^{\sqrt{x-\frac{\ell^2}{T}}}}\Lambda(n)\int_{-U- i\sqrt{T}}^{1+\epsilon- i\sqrt{T}}\left(\frac{e^{\sqrt{x-\frac{\ell^2}{T}}}}{n}\right)^s\frac{ds}{s}.
\end{align}
We  bound the integrals in RHS. Define 
\begin{align*}
    y_n(V):=\Lambda(n) \int_{-U\pm iV}^{1+\epsilon\pm iV}\sum_{\ell^2<Tx-T\log^2(n)}(-1)^\ell\left(\frac{e^{\sqrt{x-\frac{\ell^2}{T}}}}{n}\right)^s\frac{ds}{s}
\end{align*}
Inspired by  \cite{Montgomery} there exists  $\sqrt{T}<V<2\sqrt{T}$ such that
\begin{align}\label{longmont}
\sum_{n\leq e^{\sqrt{x}}}&|y_n(V)|= \sum_{n\leq e^{\sqrt{x}}}\left(\overline{y_n(V)}y_n(V)\right)^{1/2}\nonumber\\
&=\sum_{n\leq e^{\sqrt{x}}}\left(\Lambda(n)^2\left\vert \int_{-U\pm iV}^{1+\epsilon\pm iV}\sum_{\ell^2<Tx-T\log^2(n)}(-1)^\ell\left(\frac{e^{\sqrt{x-\frac{\ell^2}{T}}}}{n}\right)^s\frac{ds}{s}\right\vert^2\right)^{1/2} \nonumber\\
&\ll \sum_{n\leq e^{\sqrt{x}}}\Lambda(n)\left(\frac{1}{\sqrt{T}}\int_{\sqrt{T}}^{2\sqrt{T}}\int_{-U}^{1+\epsilon}\left\vert\sum_{\ell^2<Tx-T\log^2(n)}(-1)^\ell\left(\frac{e^{\sqrt{x-\frac{\ell^2}{T}}}}{n}\right)^{\sigma+it}\right\vert^2\frac{d\sigma}{\sigma^2+t^2} dt\right)^{1/2}\nonumber\\
&\ll \sum_{n\leq e^{\sqrt{x}}}\frac{\Lambda(n)}{\sqrt[4]{T}}\left(\int_{-U}^{1+\epsilon}\sum_{\ell_1^2<\ell_2^2<T(x-\log^2(n))}\frac{e^{\sigma(\sqrt{x-\frac{\ell_1^2}{T}}+\sqrt{x-\frac{\ell_2^2}{T}})}}{n^{2\sigma}}\left\vert\int_{\sqrt{T}}^{2\sqrt{T}}e^{it(\sqrt{x-\frac{\ell_1^2}{T}}-\sqrt{x-\frac{\ell_2^2}{T}})}\frac{dt}{\sigma^2+t^2} \right\vert d\sigma\right)^{\frac{1}{2}}\nonumber\\
&\quad\quad+\sum_{n\leq e^{\sqrt{x}}}\frac{\Lambda(n)}{\sqrt[4]{T}}\left(\sum_{\ell^2<T(x-\log^2(n))}\int_{-U}^{1+\epsilon}\frac{e^{2\sigma\sqrt{x-\frac{\ell^2}{T}}}}{n^{2\sigma}}\int_{\sqrt{T}}^{2\sqrt{T}}\frac{dt}{\sigma^2+t^2} d\sigma\right)^{\frac{1}{2}}.
\end{align}
We use Lemma \ref{vandercorput} for $G(t):=\frac{1}{\sigma^2+t^2}$ and $F(t):=t(\sqrt{x-\frac{\ell_1^2}{T}}-\sqrt{x-\frac{\ell_2^2}{T}})$ (i.e.  $F'(t)\geq \frac{\ell_2^2-\ell_1^2}{2T\sqrt{x}}$) for the off-diagonal terms in the last expression of RHS in \eqref{longmont}. Note that we could get the same result without using the lemma, but this way is more straightforward. Then 
\begin{align*}
\sum_{n\leq e^{\sqrt{x}}}|y_n(V)|&\ll \frac{\sqrt[4]{x}}{\sqrt[4]{T}}\sum_{n\leq e^{\sqrt{x}}}\Lambda(n)\left(\sum_{\ell_1^2<\ell_2^2<T(x-\log^2(n))}\int_{-U}^{1+\epsilon}\frac{e^{\sigma(\sqrt{x-\frac{\ell_1^2}{T}}+\sqrt{x-\frac{\ell_2^2}{T}})}}{n^{2\sigma}(\ell_2^2-\ell_1^2)} d\sigma\right)^{\frac{1}{2}}\nonumber\\
&+\frac{1}{\sqrt{T}}\sum_{n\leq e^{\sqrt{x}}}\Lambda(n)\left(\sum_{\ell^2<T(x-\log^2(n))}\frac{e^{2(1+\epsilon)\sqrt{x-\frac{\ell^2}{T}}}}{n^{2(1+\epsilon)}}\right)^{\frac{1}{2}}\nonumber\\
&\ll \left(\sum_{m<\sqrt{Tx}}\frac{\tau(m)}{m} \right)^{\frac{1}{2}}\frac{\sqrt[4]{x}e^{(1+\epsilon)\sqrt{x}}}{\sqrt[4]{T}}\sum_{n\leq e^{\sqrt{x}}}\frac{\Lambda(n)}{n^{1+\epsilon}}+\frac{\sqrt[4]{x}e^{(1+\epsilon)\sqrt{x}}}{\sqrt[4]{T}}\sum_{n\leq e^{\sqrt{x}}}\frac{\Lambda(n)}{n^{1+\epsilon}}
\end{align*}
where $\tau(m)$ is the number of divisors of $m$. So there exists  $\sqrt{T}< V<2\sqrt{T}$ we have
\begin{align*}
\sum_{n\leq e^{\sqrt{x}}}\Lambda(n)&\int_{-U\pm iV}^{1+\epsilon\pm iV}\sum_{\ell^2<T(x-\log^2(n))}(-1)^\ell\left(\frac{e^{\sqrt{x-\frac{\ell^2}{T}}}}{n}\right)^s\frac{ds}{s}\ll \sum_{n\leq e^{\sqrt{x}}}|y_n(V)|
\ll \frac{x^{\frac{1}{4}}e^{(1+\epsilon)\sqrt{x}}\log T}{\sqrt[4]{T}}.
\end{align*}
This and \eqref{moaadelle} imply that 
\begin{align}\label{a1}
A_1=2\pi i\sum_{l^2<Tx}(-1)^l \Psi\left(e^{\sqrt{x-\frac{l^2}{T}}}\right)+O\left(\frac{x^{\frac{1}{4}}e^{(1+\epsilon)\sqrt{x}}}{\sqrt[4]{T}}\right).
\end{align}


\begin{figure}[h!] 

\tikzset{every picture/.style={line width=0.75pt}} 

\begin{tikzpicture}[x=0.75pt,y=0.75pt,yscale=-1,xscale=1]

\draw [line width=1.5]  (282.6,2199.7) -- (541,2199.7)(308.44,2079.28) -- (308.44,2213.08) (534,2194.7) -- (541,2199.7) -- (534,2204.7) (303.44,2086.28) -- (308.44,2079.28) -- (313.44,2086.28)  ;
\draw [line width=1.5]    (263,2199.7) -- (303,2199.7) ;
\draw [line width=1.5]    (309,2199.7) -- (310,2319.7) ;
\draw  [color={rgb, 255:red, 65; green, 117; blue, 5 }  ,draw opacity=1 ][line width=1.5]  (349,2102.7) -- (507,2102.7) -- (507,2292.7) -- (349,2292.7) -- cycle ;
\draw [color={rgb, 255:red, 208; green, 2; blue, 27 }  ,draw opacity=1 ] [dash pattern={on 4.5pt off 4.5pt}]  (322.9,2087.62) -- (323.9,2306.62) ;
\draw    (358,2082.7) -- (331.86,2092.97) ;
\draw [shift={(330,2093.7)}, rotate = 338.55] [color={rgb, 255:red, 0; green, 0; blue, 0 }  ][line width=0.75]    (10.93,-3.29) .. controls (6.95,-1.4) and (3.31,-0.3) .. (0,0) .. controls (3.31,0.3) and (6.95,1.4) .. (10.93,3.29)   ;
\draw    (367,2183.7) -- (353.46,2196.34) ;
\draw [shift={(352,2197.7)}, rotate = 316.97] [color={rgb, 255:red, 0; green, 0; blue, 0 }  ][line width=0.75]    (10.93,-3.29) .. controls (6.95,-1.4) and (3.31,-0.3) .. (0,0) .. controls (3.31,0.3) and (6.95,1.4) .. (10.93,3.29)   ;
\draw    (291,2081.7) -- (302.83,2098.08) ;
\draw [shift={(304,2099.7)}, rotate = 234.16] [color={rgb, 255:red, 0; green, 0; blue, 0 }  ][line width=0.75]    (10.93,-3.29) .. controls (6.95,-1.4) and (3.31,-0.3) .. (0,0) .. controls (3.31,0.3) and (6.95,1.4) .. (10.93,3.29)   ;
\draw    (288,2297.7) -- (304.07,2293.24) ;
\draw [shift={(306,2292.7)}, rotate = 524.48] [color={rgb, 255:red, 0; green, 0; blue, 0 }  ][line width=0.75]    (10.93,-3.29) .. controls (6.95,-1.4) and (3.31,-0.3) .. (0,0) .. controls (3.31,0.3) and (6.95,1.4) .. (10.93,3.29)   ;
\draw [color={rgb, 255:red, 65; green, 117; blue, 5 }  ,draw opacity=1 ]   (442,2102.7) -- (479,2102.7) ;
\draw [shift={(481,2102.7)}, rotate = 180] [color={rgb, 255:red, 65; green, 117; blue, 5 }  ,draw opacity=1 ][line width=0.75]    (10.93,-3.29) .. controls (6.95,-1.4) and (3.31,-0.3) .. (0,0) .. controls (3.31,0.3) and (6.95,1.4) .. (10.93,3.29)   ;

\draw (456,2077.7) node [anchor=north west][inner sep=0.75pt]   [align=left] {$\displaystyle \gamma '$};
\draw (505,2182.7) node [anchor=north west][inner sep=0.75pt]  [font=\small] [align=left] {{\small U}};
\draw (252,2288.7) node [anchor=north west][inner sep=0.75pt]  [font=\footnotesize] [align=left] {$\displaystyle -\sqrt{T}$};
\draw (264,2063.7) node [anchor=north west][inner sep=0.75pt]  [font=\footnotesize] [align=left] {$\displaystyle \sqrt{T}$};
\draw (360,2165.7) node [anchor=north west][inner sep=0.75pt]   [align=left] {1+$\displaystyle \epsilon $};
\draw (312,2179.7) node [anchor=north west][inner sep=0.75pt]  [font=\footnotesize] [align=left] {$\displaystyle 0.5$};
\draw (345,2060.7) node [anchor=north west][inner sep=0.75pt]   [align=left] {Critical line};

\end{tikzpicture}

\centering \caption{The contour $\gamma'$} 
\label{Shekl4}
\end{figure}
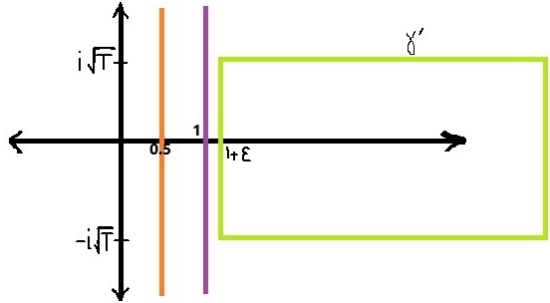


Next we compute $A_2$. We consider contour $\gamma'$ in Figure \ref{Shekl4}. As $\frac{e^{s\sqrt{x-\frac{l^2}{T}}}}{sn^s}$ does not have poles inside $\gamma'$, choosing large enough $U$ and  using the Cauchy's integral theorem give
\begin{align*}
A_2\simeq \sum_{\ell^2<Tx}(-1)^l\left(\sum_{e^{\sqrt{x-\frac{\ell^2}{T}}}\leq n}\int_{1+\epsilon+ i\sqrt{T}}^{U+ i\sqrt{T}}\frac{\Lambda(n)}{n^s}\frac{e^{s\sqrt{x-\frac{\ell^2}{T}}}}{s}ds-\sum_{e^{\sqrt{x-\frac{\ell^2}{T}}}\leq n}\int_{1+\epsilon- i\sqrt{T}}^{U- i\sqrt{T}}\frac{\Lambda(n)}{n^s}\frac{e^{s\sqrt{x-\frac{\ell^2}{T}}}}{s}ds\right).
\end{align*}
Similar to $y_n$, we define $z_n$ as follows: 
\begin{align*}
    z_n(V) := \Lambda(n)\sum_{T(x-\log^2(n))<\ell^2<Tx}(-1)^\ell \int_{1+\epsilon\pm iV}^{U\pm iV}\frac{e^{s\sqrt{x-\frac{\ell^2}{T}}}}{sn^s}ds
\end{align*}
 In this case, we will have
\begin{align}
\sum_{n}|z_n(V)|
&=\sum_{n<e^{\sqrt{x}}}\Lambda(n)\left\vert\sum_{T(x-\log^2(n))<\ell^2<Tx}(-1)^\ell \int_{1+\epsilon\pm iV}^{U\pm iV}\frac{e^{s\sqrt{x-\frac{\ell^2}{T}}}}{sn^s}ds\right\vert\nonumber\\
&+\sum_{e^{\sqrt{x}}<n}\Lambda(n)\left\vert\sum_{\ell^2<Tx}(-1)^\ell \int_{1+\epsilon\pm iV}^{U\pm iV}\frac{e^{s\sqrt{x-\frac{\ell^2}{T}}}}{sn^s}ds\right\vert.
\end{align}
As these cases are similar, we only compute the bound for the case $e^{\sqrt{x}}<n$. There exists $\sqrt{T}<V<2\sqrt{T}$ such that
\begin{align}
\sum_{e^{\sqrt{x}}<n}&|z_n(V)|\ll \frac{1}{\sqrt[4]{T}}\sum_{e^{\sqrt{x}}<n}\Lambda(n)\left(\int_{\sqrt{T}}^{2\sqrt{T}}\int_{1+\epsilon}^{U}\left\vert \sum_{\ell^2<Tx}(-1)^\ell e^{s\sqrt{x-\frac{\ell^2}{T}}}\right\vert^2 \frac{d\sigma}{n^{2\sigma}(\sigma^2+t^2)} dt\right)^{\frac{1}{2}}\nonumber\\
&\ll \frac{1}{\sqrt[4]{T}}\sum_{e^{\sqrt{x}}<n}\Lambda(n)\left(\int_{1+\epsilon}^{U}\frac{1}{n^{2\sigma}} \sum_{\ell_1^2<\ell_2^2<Tx} e^{\sigma(\sqrt{x-\frac{\ell_1^2}{T}}+\sqrt{x-\frac{\ell_2^2}{T}})}\int_{\sqrt{T}}^{2\sqrt{T}}\frac{e^{it(\sqrt{x-\frac{\ell_1^2}{T}}-\sqrt{x-\frac{\ell_2^2}{T}})}}{(\sigma^2+t^2)} dtd\sigma\right)^{\frac{1}{2}}\nonumber\\
&+ \frac{1}{\sqrt[4]{T}}\sum_{e^{\sqrt{x}}<n}\Lambda(n)\left(\int_{\sqrt{T}}^{2\sqrt{T}}\int_{1+\epsilon}^{U}\left\vert \sum_{\ell^2<Tx}(-1)^\ell e^{s\sqrt{x-\frac{\ell^2}{T}}}\right\vert^2 \frac{d\sigma}{n^{2\sigma}(\sigma^2+t^2)} dt\right)^{\frac{1}{2}}\nonumber\\
&\ll \frac{1}{\sqrt[4]{T}}\sum_{e^{\sqrt{x}}<n}\Lambda(n)\left(\int_{1+\epsilon}^{U}\frac{1}{n^{2\sigma}} \sum_{\ell_1^2<\ell_2^2<Tx} e^{\sigma(\sqrt{x-\frac{\ell_1^2}{T}}+\sqrt{x-\frac{\ell_2^2}{T}})}\int_{\sqrt{T}}^{2\sqrt{T}}\frac{e^{it(\sqrt{x-\frac{\ell_1^2}{T}}-\sqrt{x-\frac{\ell_2^2}{T}})}}{(\sigma^2+t^2)} dtd\sigma\right)^{\frac{1}{2}}\nonumber\\
&\quad\quad\quad\quad\quad\quad+\frac{1}{\sqrt[4]{T}}\sum_{e^{\sqrt{x}}<n}\Lambda(n)\left(\int_{\sqrt{T}}^{2\sqrt{T}}\sum_{\ell^2<Tx}\int_{1+\epsilon}^{U}\frac{e^{2\sigma\sqrt{x-\frac{\ell^2}{T}}}}{n^{2\sigma}}  \frac{d\sigma}{(\sigma^2+t^2)} dt\right)^{\frac{1}{2}}.
\end{align}
Let $F(t)=t(\sqrt{x-\frac{\ell_1^2}{T}}-\sqrt{x-\frac{\ell_2^2}{T}})$ and $G(t)=\frac{1}{(\sigma^2+t^2)}$. Then we conclude that $|F'(t)|\gg \frac{\ell_2^2-\ell_1^2}{T\sqrt{x}}$ and $|G(t)|\ll \frac{1}{(\sigma^2+T)}$.
Using lemma \ref{vandercorput}
\begin{align*}
\sum_{e^{\sqrt{x}}<n}|z_n(V)| \ll \frac{\sqrt[4]{x}}{\sqrt[4]{T}}&\sum_{e^{\sqrt{x}}<n}\Lambda(n)\left(\sum_{\ell_1^2<\ell_2^2<Tx}\frac{1}{\ell_2^2-\ell_1^2}  \int_{1+\epsilon}^{U}\frac{e^{\sigma(\sqrt{x-\frac{\ell_1^2}{T}}+\sqrt{x-\frac{\ell_2^2}{T}})}}{n^{2\sigma}}d\sigma\right)^{\frac{1}{2}}\nonumber\\
&+\frac{\sqrt[4]{x}}{\sqrt[4]{T}}\sum_{e^{\sqrt{x}}<n}\Lambda(n)\frac{e^{(1+\epsilon)\sqrt{x}}}{n^{1+\epsilon}}\nonumber\\
\ll\frac{\sqrt[4]{x}}{\sqrt[4]{T}}&\sum_{e^{\sqrt{x}}<n}\Lambda(n)\frac{e^{(1+\epsilon)\sqrt{x}}}{n^{1+\epsilon}}\left(\sum_{m<Tx}\frac{\tau(m)}{m} \right)^{\frac{1}{2}}\nonumber\\
&+\frac{\sqrt[4]{x}}{\sqrt[4]{T}}\sum_{e^{\sqrt{x}}<n}\Lambda(n)\frac{e^{(1+\epsilon)\sqrt{x}}}{n^{1+\epsilon}}\ll  \frac{\sqrt[4]{x}\log(T)e^{(1+\epsilon)\sqrt{x}}}{\sqrt[4]{T}}.
\end{align*}
So there exists $\sqrt{T}<V<2\sqrt{T}$ such that 
\begin{align}\label{a2}
A_2\ll\sum_{e^{\sqrt{x}}\leq n}|z_n(V)|\ll \frac{\sqrt[4]{x}\log(T)e^{(1+\epsilon)\sqrt{x}}}{\sqrt[4]{T}}.
\end{align}
Putting  \eqref{a1} and \eqref{a2} into \eqref{secondway} and comparing it with \eqref{firstway} gives
\begin{align*}
\sum_{\ell^2<Tx}(-1)^l \Psi\left(e^{\sqrt{x-\frac{\ell^2}{T}}}\right)\ll \sqrt{Tx}e^{(1+\epsilon)\sqrt{\frac{x}{1+\pi^2}}}+\frac{x^{\frac{3}{4}}}{\sqrt[4]{T}}e^{(1+\epsilon)\sqrt{x}}.
\end{align*}
Taking $T=e^{\frac{4(1+\epsilon)}{3}\sqrt{x}(1-\sqrt{\frac{1}{1+\pi^2}})}$ gives the desired result.
\qed


\bigskip


\section{Proof related to the pentagonal number theorem. }
We start this section by proving the weak pentagonal number theorem for truncation of the usual partition function.

\bigskip

We start with the proof of proposition \ref{pentcorollary}.


\begin{proof}
For \eqref{nonsensitive} we only need to put $c=\pi\sqrt{\frac{2}{3}}$, $a=\frac{3}{2}$, $b=-\frac{1}{2}$, $d=0$ in theorem \ref{main1}; for equation \eqref{pentuu}, pick $c=\frac{\pi}{\sqrt{6}}$, $a=\frac{3}{2}$, $b=-\frac{1}{2}$, $d=0$ and use Theorem \ref{main1}; and for equation \eqref{thirdfourthestimate} we need to pick $c=\frac{\pi}{\sqrt{6}}$ and $a=1$, and $b=d=0$. 

We prove equation \eqref{sensitive}.
Let $f(z)=\pi\sqrt{\frac{24(x-\frac{z(3z-1)}{2})-1}{36}}$ and $b^2=\frac{1}{12}$. 
We choose the branch cut $(-\infty,\alpha_1]\cup [\alpha_2,\infty)$. Then let $G$ be the interior of the square with vertices (see figure \ref{shekl})
 \begin{align}
\pm\sqrt{\frac{2x}{3}}\mp 1\pm ib\sqrt{x},
\end{align}
Define 
\begin{align}
h(z) :=\frac{e^{ f(z)}}{\sin^3(\pi z)} =\frac{e^{\pi \sqrt{\frac{24(x-\frac{z(3z-1)}{2})-1}{36}}}}{\sin^3(\pi z)}.
\end{align}
Using the residue theorem 
\begin{align}\label{pent2}
\int_{\gamma}&h(z)dz=2\pi i \cdot \sum_{z_i:\text{ poles }}\text{Res}(h(z_i))
\end{align}
We compute the residues of $h(z)$. We know that for $z$ near to $\ell\in \mathbb{Z}$ we have
\begin{align*}
&\frac{1}{(\sin(\pi z))^3} = \frac{(-1)^{\ell}}{\pi^3(z-\ell)^3} +\frac{(-1)^{\ell}}{2\pi (z-\ell)}+\cdots\nonumber\\
&e^{f(z)} = e^{f(\ell)} + f'(\ell)e^{f(\ell)}(z-\ell) + \frac{(f'(\ell))^2+f''(\ell)}{2}e^{f(\ell)}(z-\ell)^2+\cdots
\end{align*}
Also 
\begin{align}
    \left( f(z)\right)' = \pi (1-6z)\left(24(x-\frac{z(3z-1)}{2})-1\right)^{-1/2}\nonumber\\
    \left( f(z)\right)'' = -144\pi x \left(24(x-\frac{z(3z-1)}{2})-1\right)^{-3/2}
\end{align}
So we have
\begin{align*}
    \text{Res } h(z)|_{z=\ell} &= (-1)^{\ell}e^{ f(\ell)}\left(\frac{1}{2\pi}+\frac{(f'(\ell))^2+f''(\ell)}{2\pi^3}\right)\nonumber\\
    &= \frac{(-1)^{\ell}e^{ f(\ell)}}{2\pi(24(x-G_{\ell})-1)} \left(24(x-G_{\ell})-1+(1-6\ell)^2-\frac{144x}{\pi \sqrt{24(x-G_{\ell})-1}}\right)\nonumber\\
    &=\frac{(-1)^{\ell}e^{ f(\ell)}}{2\pi(24(x-G_{\ell})-1)} \left(24x-\frac{144x}{\pi \sqrt{24(x-G_{\ell})-1}}\right)
\end{align*}
It implies that
\begin{align}
\int_{\gamma}h(z)dz&=24ix\sum_{G_l<x}(-1)^l\frac{e^{\frac{\pi}{6}\sqrt{24(x-G_l)-1}}}{(24(x-G_l)-1)}\left(1-\frac{6}{\pi\sqrt{24(x-G_l)-1}}\right)\nonumber\\
&=\frac{24ix}{\sqrt{12}}\sum_{G_l<x}(-1)^lp_2(x-G_l).
\end{align}
We bound the integral. First assume that we choose $z\in \gamma_1\cup\gamma_3$. So $z=t\pm ib\sqrt{x}$ for $-\sqrt{\frac{2x}{3}}+1<t<\sqrt{\frac{2x}{3}}-1$. For large enough $x$ we have
\begin{align*}
\frac{24(x-\frac{z(3z-1)}{2})-1}{36}
\sim \frac{2}{3}x-t^2+b^2x\mp 2ibt\sqrt{x}.
\end{align*}
Similar to the proof of theorem \ref{main1} for $z\in \gamma_1,\gamma_3$
\begin{align*}
 e^{\pi\sqrt{\frac{24\left(x-\frac{z(3z-1)}{2}\right)-1}{36}}}\leq e^{\pi\sqrt{\frac{2x}{3}+b^2x}}.
\end{align*}
Also $|\sin^3(\pi z)| \sim \frac{1}{8}e^{3\pi b\sqrt{x}}$, and considering $b^2=\frac{1}{12}$ we get that
\begin{align}\label{bavali}
\pi(\sqrt{\frac{2}{3}+b^2}-3b) = 0.
\end{align}
So the contribution of the horizontal legs is at most $o(x)$. Now we compute the case $z\in \gamma_2,\gamma_4$. We have $z=\pm\sqrt{\frac{2}{3}x}\mp 1+it$ and $-b\sqrt{x}<t<b\sqrt{x}$. We have
\begin{align*}
\frac{24(x-\frac{z(3z-1)}{2})-1}{36}
=\frac{2}{3}x-\frac{2}{3}\times\frac{2x\mp 2it\sqrt{6x}-3t^2\mp 7it -(2\sqrt{6}+\sqrt{\frac{2}{3}})\sqrt{x}+4}{2}-\frac{2}{3} \nonumber\\
\end{align*}
If $t=o(\sqrt{x})$, then $\frac{24(x-\frac{z(3z-1)}{2})-1}{36}=o(x)$. Otherwise, since $\sqrt{x},t$ are negligible in comparison to $x,t^2$
\begin{align*}
\frac{24(x-\frac{z(3z-1)}{2})-1}{36}\sim t^2\mp it\sqrt{\frac{8x}{3}}.
\end{align*} 
Using lemma \ref{l:sqrt} we have
\begin{align}\label{longu}
Re\left(\sqrt{\frac{24(x-\frac{z(3z-1)}{2})-1}{36}}\right)
&\leq \sqrt{\frac{t}{2\sqrt{3}}\left(\sqrt{3t^2+8x}+t\sqrt{3}\right)}
\end{align}
Hence in any case we get
\begin{align}\label{rightgamma}
\left\vert \frac{e^{\pi \sqrt{\frac{24(x-\frac{z(3z-1)}{2})-1}{36}}}}{(\sin (\pi z))^3}\right\vert\ll e^{\pi \sqrt{\frac{t}{2\sqrt{3}}\left(\sqrt{3t^2+8}+t\sqrt{3}\right)}-3\pi t}.
\end{align}
Maximizing for $t$ in the RHS of \eqref{rightgamma}, we get that the integral in the LHS of \eqref{pent2} can be at most $e^{0.21\sqrt{x}}$. 
It completes the proof.
\end{proof}


\begin{remark}
Perhaps the upper bound in equation  \eqref{sensitive} becomes smaller by taking higher powers of $\sin(\pi z)$. However, this argument needs a new nontrivial input to construct the proof of the Pentagonal Number Theorem. In other word, one may take the function
\begin{align}
    H(z) := \sum_{k\in \mathbb{Z}}\sum_{\substack{0\leq h<k\\(h,k)=1}} \frac{\sinh\left(\frac{\pi}{6}\sqrt{24(x-\frac{z(3z-1)}{2})-1}\right)y(x,z,h,k)}{(\sin(\pi z))^{2r+1}}
\end{align}
for suitable function $y(x,z,h,k)$ and integer $r$ to generate all of the terms in the Ramanujan-Hardy-Rademacher formula \eqref{RHR} for the partition function. One may use this possible nontrivial observation to  prove that all of the error terms from different $k,h$ will cancel each other, which is equivalent to the Pentagonal Number Theorem.
\end{remark}

\bigskip


\textit{Proof of Proposition \ref{pack}. } We have
\begin{align*}
\sum_{\ell^2<x}&(-1)^{\ell} p_3(x-\ell^2)=\sqrt{6}e^{\pi i x}\sum_{\ell^2<x}(\frac{1}{24(x-\ell^2)-1}-\frac{12}{\pi(24(x-\ell^2)-1)^{\frac{3}{2}}})e^{\frac{\pi}{12}\sqrt{24(x-\ell^2)-1}}\nonumber\\
&=\sqrt{6}e^{\frac{\pi}{12}\sqrt{24x-1}+\pi i x}\sum_{\ell^2<\frac{x}{4}}e^{\frac{-2\pi \ell^2}{\sqrt{24x-1}(\sqrt{1-\frac{24\ell^2}{24x-1}}+1)}}(\frac{1}{24(x-\ell^2)-1}+O(\frac{1}{\sqrt{x^3}}))\\
&\sim\ \frac{e^{\frac{\pi}{12}\sqrt{24x - 1} + \pi i x}}{4x \sqrt{6}} 
\sum_{\ell^2 < \sqrt{x} \ln x} e^{-\ell^2/2\sigma^2},
\end{align*}
where 
$$
\sigma^2\ =\ {\sqrt{6x} \over \pi}.
$$
The last expression in the above can be approximated as follows
\begin{eqnarray*}
\frac{1}{4 x \sqrt{6}}e^{\frac{\pi}{12}\sqrt{24x - 1} + \pi i x} 
\sum_{\ell^2 < \sqrt{x} \ln x} e^{-\frac{\ell^2}{2\sigma^2}}\ &\sim&\ 
\frac{1}{4x \sqrt{6}} e^{\frac{\pi}{12}\sqrt{24x-1} + \pi i x} 
\int_{-\infty}^\infty e^{-t^2/2\sigma^2} dt\\ 
&=&\ \frac{\sigma \sqrt{2\pi}}{4x\sqrt{6}} e^{\frac{\pi}{12}\sqrt{24x-1} + \pi i x}\\
&\sim&\ {e^{\pi ix}\over 2^{3/4}x^{1/4}}\sqrt{p(x)}.
\end{eqnarray*}
\qed


\bigskip

We need the next lemma. 


\begin{lemma}\label{lemmo}
With the same notation as theorem \ref{main1}
\begin{align}\label{pentagonalbessel}
 \sum_{n:an^2+bn+d<x}(-1)^n I_{\alpha}\left(c\sqrt{x-an^2+bn+d}\right)h(n)=O\left(e^{cw\sqrt{x}}\right).
\end{align}
where $I_{\alpha}$ is the Bessel function.
\end{lemma}


Before we mention the proof note that for fixed $\alpha$ and large enough $x$
$$I_{\alpha}(x) \sim \frac{e^{x}}{\sqrt{2\pi x}} \left(1 + O\left(\frac{1}{x}\right)\right). $$ 


\begin{proof}
Since the proof is very similar to proof of \ref{main1}, we skip the details. Let $H(z)=\frac{h(z)I_{\alpha}(\sqrt{x-z^2})}{\sin(\pi z)}$ and $q(n)=an^2+bn+d$ and assume the contour $\gamma$ in \ref{shekl}. Then
\begin{align*}
\sum_{\ell: q(\ell)<x}(-1)^\ell I_{\alpha}\left(c\sqrt{x-q(\ell)}\right)h(\ell)=\int_{\gamma}H(z)dz
\end{align*}
For $z\in \gamma_1,\gamma_3$ 
\begin{align*}
|H(z)|\ll \frac{I_{\alpha}\left(c\sqrt{x(1+au^2)}\right)}{e^{\pi u\sqrt{x}}}\ll e^{c\sqrt{x(1+au^2)}-\pi u\sqrt{x}}.
\end{align*}
Also for $z\in \gamma_2,\gamma_4$
\begin{align*}
|H(z)|\ll I_{\alpha}\left(c\sqrt{\frac{ux}{2}\left(\sqrt{au^2+4}+u\sqrt{a}\right)}\right)\ll \sqrt{x}e^{c\sqrt{x\sqrt{a\alpha}\frac{\sqrt{a\alpha^2+4}+\alpha\sqrt{a}}{2}}-\pi \alpha\sqrt{x}}.
\end{align*}
with the same notation as in proof of theorem \ref{main1}. As the bound of argument of Bessel function is the same as exponents in the proof of theorem \ref{main1} we get the same bound.
\end{proof}


\bigskip


\textit{Proof of corollaries \ref{penbessel1} and \ref{penbessel2}. } 
For corollary \ref{penbessel1} pick $a=1$, and $c=\sqrt{\frac{2\pi^2}{3}}$ in the Lemma \ref{lemmo}. For corollary \ref{penbessel2} pick $c=\frac{2\pi}{ \sqrt{15}}$  and $a=1$ in Lemma \ref{lemmo}. 
\qed


\bigskip


\section{Proof related to the Prouhet-Tarry-Escott problem}\label{PTE}

It is worth establishing a ``baseline result" related to problem \ref{terryescottapproximation} for $N$ large, relative to
$k, n$, that we get easily from a Pigeonhole Argument: consider all vectors
$(x, x^2,\cdots, x^k)$ with $1 \leq x \leq N$. The sum of $n$ of these lie in a box of volume
$n^kN^{k(k+1)/2}$; and if two such sums belong to the same box with dimensions
$N^c \times N^{2c}\times \cdots\times N^{kc}$, then they give a solution to \eqref{terryescottapproximation} for all $1\leq i\leq k$. The
number of non over-lapping $N^c\times \cdots\times N^{kc}$ boxes that fit inside our volume $n^kN^{k(k+1)/2}$ is at
most $n^kN^{(1-c)k(k+1)/2}$; and with a little work one can see that the large box
can be covered with approximately (up to a constant factor) this many smaller
boxes. If this (the number of smaller boxes in a covering) is smaller than the
number of sets of $n$ vectors $(x, x^2, \cdots, x^k)$ that produce our vector sum  (this
count is at least $ \frac{N^n}{n!}$ for $N$ large enough relative to $n$) then we get a ``collision",
that is a pair of sequences $a_1, \cdots, a_n$ and $b_1, \cdots, b_n$ leading to a solution to \ref{terryescottapproximation}
for all $1\leq i \leq k$. In other words, we get such a solution when
\begin{align*}
n^kN^{\frac{(1-c)k(k+1)}{2}}<\frac{N^n}{n!}.
\end{align*}
For $N$ large, then, we get that there is a solution so long as
\begin{align}
c> 1-\frac{2n}{k(k+1)}.
\end{align}
When $k$ is a little smaller than
$\sqrt{2n}$, note that the RHS is negative,
implying that we can take $c = 0$ (since it must be non-negative).

Curiously, when $k$ is only a little bigger than $\sqrt{n}$ (say, $\sqrt{n}\log(n)$), then this pigeonhole argument only gives us pairs of sequences with $c$ near to $1$.
Basically, then, we don't get a much better result for the weakening than we do for the original Prouhet-Tarry-Escott Problem, if we insist on finding solutions with $c < \frac{1}{2}$, say.

\bigskip

We prove a lemma before introducing a set of solutions for the weak Prouhet-Terry-Escott problem (problem \ref{terryescottapproximation}).


\begin{lemma}\label{lem}
For large $x$, let $k\ll \frac{\sqrt{x}}{\log(xT)}$ and $T:=T(x)=o(x)$. Then for every  $1\leq r\leq k$ there exists $c>0$ such that 
\begin{align}
\sum_{\ell^2<xT}(-1)^\ell \left(xT-\ell^2\right)^{\frac{r}{2}}\ll \sqrt{x}(Tx)^{\frac{r}{4}} (Ar)^{r/2}
\end{align}
\end{lemma}



\begin{remark}
Note that the proof becomes easier if we just choose $r$ to be even. But we propose a more general case here.
\end{remark}



\begin{proof}
Let $u=o(\sqrt{x})$, to be determined later. Define
\begin{align*}
f_r(z)=\frac{\left(xT-z^2\right)^{\frac{r}{2}}}{\sin(\pi z)}.
\end{align*}
Let $\gamma$ be the contour in Figure \ref{shekl3}. Using the residue Theorem 
\begin{align*}
\int_{\gamma}f_r(z)dz= 2\pi i\sum_{\ell^2<xT}(-1)^\ell\left(Tx-\ell^2\right)^{\frac{r}{2}}.
\end{align*}


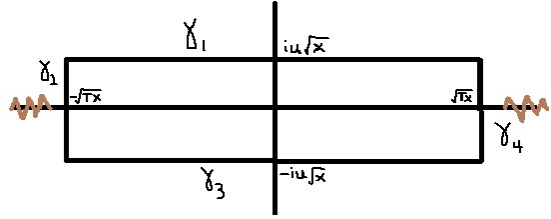
\begin{figure}[h!] 

\tikzset{every picture/.style={line width=0.75pt}} 

\begin{tikzpicture}[x=0.75pt,y=0.75pt,yscale=-1,xscale=1]

\draw [line width=1.5]  (300.98,2512.7) -- (537.2,2512.7)(324.6,2393) -- (324.6,2526) (530.2,2507.7) -- (537.2,2512.7) -- (530.2,2517.7) (319.6,2400) -- (324.6,2393) -- (329.6,2400)  ;
\draw [line width=1.5]    (124,2510.7) -- (301,2512.7) ;
\draw [line width=1.5]    (325,2522.7) -- (326.3,2600.18) -- (326,2609.7) ;
\draw   (160.1,2446.2) -- (489.1,2446.2) -- (489.1,2578.2) -- (160.1,2578.2) -- cycle ;
\draw [color={rgb, 255:red, 208; green, 2; blue, 27 }  ,draw opacity=1 ][line width=1.5]    (116,2510.7) -- (153,2510.7) ;
\draw  [color={rgb, 255:red, 208; green, 2; blue, 27 }  ,draw opacity=1 ][line width=3] [line join = round][line cap = round] (236,2509.7) .. controls (236,2510.04) and (236,2510.37) .. (236,2510.7) ;
\draw  [color={rgb, 255:red, 208; green, 2; blue, 27 }  ,draw opacity=1 ][line width=3] [line join = round][line cap = round] (243,2509.7) .. controls (243,2509.7) and (243,2509.7) .. (243,2509.7) ;
\draw  [color={rgb, 255:red, 208; green, 2; blue, 27 }  ,draw opacity=1 ][line width=3] [line join = round][line cap = round] (250,2509.7) .. controls (250,2509.7) and (250,2509.7) .. (250,2509.7) ;
\draw  [color={rgb, 255:red, 208; green, 2; blue, 27 }  ,draw opacity=1 ][line width=3] [line join = round][line cap = round] (392,2510.7) .. controls (392,2510.7) and (392,2510.7) .. (392,2510.7) ;
\draw  [color={rgb, 255:red, 208; green, 2; blue, 27 }  ,draw opacity=1 ][line width=3] [line join = round][line cap = round] (400,2510.7) .. controls (400,2511.04) and (400,2511.37) .. (400,2511.7) ;
\draw  [color={rgb, 255:red, 208; green, 2; blue, 27 }  ,draw opacity=1 ][line width=3] [line join = round][line cap = round] (408,2510.7) .. controls (408,2511.04) and (408,2511.37) .. (408,2511.7) ;
\draw    (310,2577.7) -- (444,2578.69) ;
\draw [shift={(446,2578.7)}, rotate = 180.42] [color={rgb, 255:red, 0; green, 0; blue, 0 }  ][line width=0.75]    (10.93,-3.29) .. controls (6.95,-1.4) and (3.31,-0.3) .. (0,0) .. controls (3.31,0.3) and (6.95,1.4) .. (10.93,3.29)   ;
\draw    (523,2541.7) -- (495.54,2518.98) ;
\draw [shift={(494,2517.7)}, rotate = 399.61] [color={rgb, 255:red, 0; green, 0; blue, 0 }  ][line width=0.75]    (10.93,-3.29) .. controls (6.95,-1.4) and (3.31,-0.3) .. (0,0) .. controls (3.31,0.3) and (6.95,1.4) .. (10.93,3.29)   ;
\draw    (133,2475.7) -- (154.76,2503.14) ;
\draw [shift={(156,2504.7)}, rotate = 231.57999999999998] [color={rgb, 255:red, 0; green, 0; blue, 0 }  ][line width=0.75]    (10.93,-3.29) .. controls (6.95,-1.4) and (3.31,-0.3) .. (0,0) .. controls (3.31,0.3) and (6.95,1.4) .. (10.93,3.29)   ;
\draw    (368,2447.2) -- (233,2446.22) ;
\draw [shift={(231,2446.2)}, rotate = 360.41999999999996] [color={rgb, 255:red, 0; green, 0; blue, 0 }  ][line width=0.75]    (10.93,-3.29) .. controls (6.95,-1.4) and (3.31,-0.3) .. (0,0) .. controls (3.31,0.3) and (6.95,1.4) .. (10.93,3.29)   ;
\draw    (171,2386.2) ;
\draw [color={rgb, 255:red, 208; green, 2; blue, 27 }  ,draw opacity=1 ]   (476,2513.2) ;
\draw [shift={(476,2513.2)}, rotate = 0] [color={rgb, 255:red, 208; green, 2; blue, 27 }  ,draw opacity=1 ][fill={rgb, 255:red, 208; green, 2; blue, 27 }  ,fill opacity=1 ][line width=0.75]      (0, 0) circle [x radius= 3.35, y radius= 3.35]   ;
\draw [color={rgb, 255:red, 208; green, 2; blue, 27 }  ,draw opacity=1 ][line width=1.5]    (505,2512.2) -- (532,2512.2) ;
\draw [color={rgb, 255:red, 208; green, 2; blue, 27 }  ,draw opacity=1 ]   (456,2513.2) ;
\draw [shift={(456,2513.2)}, rotate = 0] [color={rgb, 255:red, 208; green, 2; blue, 27 }  ,draw opacity=1 ][fill={rgb, 255:red, 208; green, 2; blue, 27 }  ,fill opacity=1 ][line width=0.75]      (0, 0) circle [x radius= 3.35, y radius= 3.35]   ;
\draw [color={rgb, 255:red, 208; green, 2; blue, 27 }  ,draw opacity=1 ]   (434,2513.2) ;
\draw [shift={(434,2513.2)}, rotate = 0] [color={rgb, 255:red, 208; green, 2; blue, 27 }  ,draw opacity=1 ][fill={rgb, 255:red, 208; green, 2; blue, 27 }  ,fill opacity=1 ][line width=0.75]      (0, 0) circle [x radius= 3.35, y radius= 3.35]   ;
\draw [color={rgb, 255:red, 208; green, 2; blue, 27 }  ,draw opacity=1 ]   (372,2512.2) ;
\draw [shift={(372,2512.2)}, rotate = 0] [color={rgb, 255:red, 208; green, 2; blue, 27 }  ,draw opacity=1 ][fill={rgb, 255:red, 208; green, 2; blue, 27 }  ,fill opacity=1 ][line width=0.75]      (0, 0) circle [x radius= 3.35, y radius= 3.35]   ;
\draw [color={rgb, 255:red, 208; green, 2; blue, 27 }  ,draw opacity=1 ]   (349,2512.2) ;
\draw [shift={(349,2512.2)}, rotate = 0] [color={rgb, 255:red, 208; green, 2; blue, 27 }  ,draw opacity=1 ][fill={rgb, 255:red, 208; green, 2; blue, 27 }  ,fill opacity=1 ][line width=0.75]      (0, 0) circle [x radius= 3.35, y radius= 3.35]   ;
\draw [color={rgb, 255:red, 208; green, 2; blue, 27 }  ,draw opacity=1 ]   (324.6,2512.7) ;
\draw [shift={(324.6,2512.7)}, rotate = 0] [color={rgb, 255:red, 208; green, 2; blue, 27 }  ,draw opacity=1 ][fill={rgb, 255:red, 208; green, 2; blue, 27 }  ,fill opacity=1 ][line width=0.75]      (0, 0) circle [x radius= 3.35, y radius= 3.35]   ;
\draw [color={rgb, 255:red, 208; green, 2; blue, 27 }  ,draw opacity=1 ]   (301,2512.7) ;
\draw [shift={(301,2512.7)}, rotate = 0] [color={rgb, 255:red, 208; green, 2; blue, 27 }  ,draw opacity=1 ][fill={rgb, 255:red, 208; green, 2; blue, 27 }  ,fill opacity=1 ][line width=0.75]      (0, 0) circle [x radius= 3.35, y radius= 3.35]   ;
\draw [color={rgb, 255:red, 208; green, 2; blue, 27 }  ,draw opacity=1 ]   (281,2512.2) ;
\draw [shift={(281,2512.2)}, rotate = 0] [color={rgb, 255:red, 208; green, 2; blue, 27 }  ,draw opacity=1 ][fill={rgb, 255:red, 208; green, 2; blue, 27 }  ,fill opacity=1 ][line width=0.75]      (0, 0) circle [x radius= 3.35, y radius= 3.35]   ;
\draw [color={rgb, 255:red, 208; green, 2; blue, 27 }  ,draw opacity=1 ]   (212.5,2511.7) ;
\draw [shift={(212.5,2511.7)}, rotate = 0] [color={rgb, 255:red, 208; green, 2; blue, 27 }  ,draw opacity=1 ][fill={rgb, 255:red, 208; green, 2; blue, 27 }  ,fill opacity=1 ][line width=0.75]      (0, 0) circle [x radius= 3.35, y radius= 3.35]   ;
\draw [color={rgb, 255:red, 208; green, 2; blue, 27 }  ,draw opacity=1 ]   (193,2511.2) ;
\draw [shift={(193,2511.2)}, rotate = 0] [color={rgb, 255:red, 208; green, 2; blue, 27 }  ,draw opacity=1 ][fill={rgb, 255:red, 208; green, 2; blue, 27 }  ,fill opacity=1 ][line width=0.75]      (0, 0) circle [x radius= 3.35, y radius= 3.35]   ;
\draw [color={rgb, 255:red, 208; green, 2; blue, 27 }  ,draw opacity=1 ]   (172,2511.2) ;
\draw [shift={(172,2511.2)}, rotate = 0] [color={rgb, 255:red, 208; green, 2; blue, 27 }  ,draw opacity=1 ][fill={rgb, 255:red, 208; green, 2; blue, 27 }  ,fill opacity=1 ][line width=0.75]      (0, 0) circle [x radius= 3.35, y radius= 3.35]   ;

\draw (76,2449.7) node [anchor=north west][inner sep=0.75pt]  [xslant=-0.04] [align=left] {\mbox{-}$\displaystyle \sqrt{Tx}$};
\draw (516,2530.7) node [anchor=north west][inner sep=0.75pt]  [xslant=-0.04] [align=left] {$\displaystyle \sqrt{Tx}$};
\draw (326,2574.7) node [anchor=north west][inner sep=0.75pt]   [align=left] {$\displaystyle -u\sqrt{x}$};
\draw (326,2413.7) node [anchor=north west][inner sep=0.75pt]   [align=left] {$\displaystyle u\sqrt{x}$};
\draw (488.75,2445.9) node [anchor=north west][inner sep=0.75pt]   [align=left] {$\displaystyle \gamma _{4}$};
\draw (441,2576.7) node [anchor=north west][inner sep=0.75pt]   [align=left] {$\displaystyle \gamma _{3}$};
\draw (137,2546.7) node [anchor=north west][inner sep=0.75pt]   [align=left] {$\displaystyle \gamma _{2}$};
\draw (207,2416.7) node [anchor=north west][inner sep=0.75pt]   [align=left] {$\displaystyle \gamma _{1}$};

\end{tikzpicture}

\centering \caption{The contour $\gamma$} 
\label{shekl3}
\end{figure}

Let $z\in \gamma_1,\gamma_3$. So $z=t\pm iu\sqrt{x}$ and $-\sqrt{xT}<t<\sqrt{xT}$. Then 
\begin{align*}
\left\vert xT-z^2\right\vert^2
=\ &\left(xT+u^2x-t^2\right)^2+ 4t^2u^2x.
\end{align*}
Note that $u$ is a constant as $x$ tends to infinity and $T=o(x)$. With simple computation we can conclude that the RHS is maximaized at $t=0$, so 
on $\gamma_1,\gamma_3$ we have 
\begin{align}\label{gamaa1}
|f_r(z)|\ll \left\vert xT+u^2x\right\vert^{\frac{r}{4}}e^{-\pi u\sqrt{x}}\sim (xT)^{\frac{r}{4}} e^{-\pi u\sqrt{x}}.
\end{align}
By assumption $r\leq k\ll \sqrt{x}/\log(xT)$, so we can pick $u$ to be large enough so as the contribution of horizental legs become small. For $\gamma_2,\gamma_4$ we have $z=\pm \sqrt{xT}+it$ and $-u\sqrt{x}<t<u\sqrt{x}$. We can show that
\begin{align*}
\left\vert xT-z^2\right\vert^2
=t^4+4t^2xT.
\end{align*}
So we need to maximize the RHS of the following expression for $t\leq u\sqrt{x}$  
$$|f_r(z)|\ll (t^4+4t^2xT)^{\frac{r}{4}}e^{-\pi t}$$
Simple computation shows that it happens when $t\sim Cr$ for some $C>0$. Hence, there exist $A>0$ such that
$$|f_r(z)|\ll A^r (xT)^{\frac{r}{4}} \left(r^2 + \frac{r^4}{xT}\right)^{\frac{r}{4}} \ll (Ar)^{r/2} (xT)^{\frac{r}{4}}.$$
This completes the proof.
\end{proof}


\bigskip

\begin{remark}
We could increase the height of vertical lines of figure \ref{shekl3} to $x^{\alpha}$, $\alpha>\frac{1}{2}$, to make it possible for $k$ to become bigger - say $k\gg x^{\alpha}$. This in turn results in larger $k=M(n)$ and larger error term.   
\end{remark}

\bigskip


\textit{Proof of Theorem \ref{pte0}. } Let $M$ be a large number. 
\begin{align*}
x_i=M^{2m+b}-(2i-2)^2\quad \quad \quad
y_i=M^{2m+b}-(2i-1)^2
\end{align*}
Then  $\max(x_i^r,y_i^r)\sim M^{2m+b}$. Lemma \ref{lem} concludes that for $x=M^{2m}$ and $T=M^{b}$ and $1\leq r\leq k$
\begin{align*}
\sum_{i}x_i^r-\sum_i y_i^r \ll (r)^{\frac{r}{2}+\epsilon}M^{(2m+b)\frac{r}{4}+m}.
\end{align*}
If we pick $k\leq \frac{u\pi M^{m}}{12m^2\log(M)}$ and $b=1$, then the result follows.
\qed
        

\bigskip


\textit{Proof of Theorem \ref{polynomialcase}. } 
We first show that $f_r(M)$ is a polynomial in $M$ -- that is, 
$$f_r(M) = c_0(r) + c_1(r)M + \cdots + c_d(r) M^d,$$ 
where $d$ is yet to be determined.  This follows upon applying the binomial theorem to the terms in the definition of $f_r(M)$, together with the fact that $\sum_{|\ell| < 2M} (-1)^{\ell} \ell^k$ is a polynomial in $M$. The coefficients are obviously integers and we also can show the coefficients as sums involving Bernouli numbers. Note that the degree $d$ of that polynomial doesn't depend on $M$.  

Let's assume that $r$ is even. We now leverage this fact to prove that $d = r-1$.  To do this, note that it suffices to prove that $|f_r(M)| = o_r( M^r)$, and $|f_r(M)| \gg_r M^{r-1}$. 
To put that another way:  fix $r$, and then we show that 
$$ \lim_{M \to \infty} \frac{\log( |f_r(M)|)}{\log(M)} = r-1. $$
 Write $f_r(M)$ as the contour integral 
\begin{align*}
\frac{1}{2\pi i} \int_{\gamma} f(z) dz:=\frac{1}{2\pi i} \int_{\gamma} \frac{(4M^2 - z^2)^r}{\sin(\pi z)} dz,
\end{align*} 
where $\gamma$ is in figure \ref{sheklaku}. Note that because $f$ has a removable singularity at $z=\pm 2M$, it is possible to compute the contribution of the integral in these vertical legs.


\begin{figure}[h!]

\tikzset{every picture/.style={line width=0.75pt}} 

\begin{tikzpicture}[x=0.65pt,y=0.65pt,yscale=-0.8,xscale=0.8]

\draw [line width=1.5]  (300.67,2932.1) -- (509,2932.1)(321.5,2769.7) -- (321.5,2950.15) (502,2927.1) -- (509,2932.1) -- (502,2937.1) (316.5,2776.7) -- (321.5,2769.7) -- (326.5,2776.7)  ;
\draw [line width=1.5]    (144.58,2929.39) -- (300.69,2932.1) ;
\draw [line width=1.5]    (321.85,2945.67) -- (323,3050.78) -- (322.74,3063.7) ;
\draw  [line width=2.25]  (176.42,2841.88) -- (466.58,2841.88) -- (466.58,3020.97) -- (176.42,3020.97) -- cycle ;
\draw [color={rgb, 255:red, 208; green, 2; blue, 27 }  ,draw opacity=1 ][line width=1.5]    (137.53,2929.39) -- (170.16,2929.39) ;
\draw  [color={rgb, 255:red, 208; green, 2; blue, 27 }  ,draw opacity=1 ][line width=3] [line join = round][line cap = round] (243.36,2928.03) .. controls (243.36,2928.48) and (243.36,2928.94) .. (243.36,2929.39) ;
\draw  [color={rgb, 255:red, 208; green, 2; blue, 27 }  ,draw opacity=1 ][line width=3] [line join = round][line cap = round] (249.53,2928.03) .. controls (249.53,2928.03) and (249.53,2928.03) .. (249.53,2928.03) ;
\draw  [color={rgb, 255:red, 208; green, 2; blue, 27 }  ,draw opacity=1 ][line width=3] [line join = round][line cap = round] (255.71,2928.03) .. controls (255.71,2928.03) and (255.71,2928.03) .. (255.71,2928.03) ;
\draw  [color={rgb, 255:red, 208; green, 2; blue, 27 }  ,draw opacity=1 ][line width=3] [line join = round][line cap = round] (380.94,2929.39) .. controls (380.94,2929.39) and (380.94,2929.39) .. (380.94,2929.39) ;
\draw  [color={rgb, 255:red, 208; green, 2; blue, 27 }  ,draw opacity=1 ][line width=3] [line join = round][line cap = round] (388,2929.39) .. controls (388,2929.84) and (388,2930.29) .. (388,2930.75) ;
\draw  [color={rgb, 255:red, 208; green, 2; blue, 27 }  ,draw opacity=1 ][line width=3] [line join = round][line cap = round] (395.05,2929.39) .. controls (395.05,2929.84) and (395.05,2930.29) .. (395.05,2930.75) ;
\draw    (308.62,3019.45) -- (426.57,3020.78) ;
\draw [shift={(428.57,3020.81)}, rotate = 180.65] [color={rgb, 255:red, 0; green, 0; blue, 0 }  ][line width=0.75]    (10.93,-3.29) .. controls (6.95,-1.4) and (3.31,-0.3) .. (0,0) .. controls (3.31,0.3) and (6.95,1.4) .. (10.93,3.29)   ;
\draw    (359.78,2841.56) -- (240.95,2841.87) ;
\draw [shift={(238.95,2841.88)}, rotate = 359.85] [color={rgb, 255:red, 0; green, 0; blue, 0 }  ][line width=0.75]    (10.93,-3.29) .. controls (6.95,-1.4) and (3.31,-0.3) .. (0,0) .. controls (3.31,0.3) and (6.95,1.4) .. (10.93,3.29)   ;
\draw [color={rgb, 255:red, 208; green, 2; blue, 27 }  ,draw opacity=1 ]   (455.03,2932.78) ;
\draw [shift={(455.03,2932.78)}, rotate = 0] [color={rgb, 255:red, 208; green, 2; blue, 27 }  ,draw opacity=1 ][fill={rgb, 255:red, 208; green, 2; blue, 27 }  ,fill opacity=1 ][line width=0.75]      (0, 0) circle [x radius= 3.35, y radius= 3.35]   ;
\draw [color={rgb, 255:red, 208; green, 2; blue, 27 }  ,draw opacity=1 ][line width=1.5]    (480.6,2931.42) -- (504.42,2931.42) ;
\draw [color={rgb, 255:red, 208; green, 2; blue, 27 }  ,draw opacity=1 ]   (437.39,2932.78) ;
\draw [shift={(437.39,2932.78)}, rotate = 0] [color={rgb, 255:red, 208; green, 2; blue, 27 }  ,draw opacity=1 ][fill={rgb, 255:red, 208; green, 2; blue, 27 }  ,fill opacity=1 ][line width=0.75]      (0, 0) circle [x radius= 3.35, y radius= 3.35]   ;
\draw [color={rgb, 255:red, 208; green, 2; blue, 27 }  ,draw opacity=1 ]   (417.99,2932.78) ;
\draw [shift={(417.99,2932.78)}, rotate = 0] [color={rgb, 255:red, 208; green, 2; blue, 27 }  ,draw opacity=1 ][fill={rgb, 255:red, 208; green, 2; blue, 27 }  ,fill opacity=1 ][line width=0.75]      (0, 0) circle [x radius= 3.35, y radius= 3.35]   ;
\draw [color={rgb, 255:red, 208; green, 2; blue, 27 }  ,draw opacity=1 ]   (363.3,2931.42) ;
\draw [shift={(363.3,2931.42)}, rotate = 0] [color={rgb, 255:red, 208; green, 2; blue, 27 }  ,draw opacity=1 ][fill={rgb, 255:red, 208; green, 2; blue, 27 }  ,fill opacity=1 ][line width=0.75]      (0, 0) circle [x radius= 3.35, y radius= 3.35]   ;
\draw [color={rgb, 255:red, 208; green, 2; blue, 27 }  ,draw opacity=1 ]   (343.02,2931.42) ;
\draw [shift={(343.02,2931.42)}, rotate = 0] [color={rgb, 255:red, 208; green, 2; blue, 27 }  ,draw opacity=1 ][fill={rgb, 255:red, 208; green, 2; blue, 27 }  ,fill opacity=1 ][line width=0.75]      (0, 0) circle [x radius= 3.35, y radius= 3.35]   ;
\draw [color={rgb, 255:red, 208; green, 2; blue, 27 }  ,draw opacity=1 ]   (321.5,2932.1) ;
\draw [shift={(321.5,2932.1)}, rotate = 0] [color={rgb, 255:red, 208; green, 2; blue, 27 }  ,draw opacity=1 ][fill={rgb, 255:red, 208; green, 2; blue, 27 }  ,fill opacity=1 ][line width=0.75]      (0, 0) circle [x radius= 3.35, y radius= 3.35]   ;
\draw [color={rgb, 255:red, 208; green, 2; blue, 27 }  ,draw opacity=1 ]   (300.69,2932.1) ;
\draw [shift={(300.69,2932.1)}, rotate = 0] [color={rgb, 255:red, 208; green, 2; blue, 27 }  ,draw opacity=1 ][fill={rgb, 255:red, 208; green, 2; blue, 27 }  ,fill opacity=1 ][line width=0.75]      (0, 0) circle [x radius= 3.35, y radius= 3.35]   ;
\draw [color={rgb, 255:red, 208; green, 2; blue, 27 }  ,draw opacity=1 ]   (283.05,2931.42) ;
\draw [shift={(283.05,2931.42)}, rotate = 0] [color={rgb, 255:red, 208; green, 2; blue, 27 }  ,draw opacity=1 ][fill={rgb, 255:red, 208; green, 2; blue, 27 }  ,fill opacity=1 ][line width=0.75]      (0, 0) circle [x radius= 3.35, y radius= 3.35]   ;
\draw [color={rgb, 255:red, 208; green, 2; blue, 27 }  ,draw opacity=1 ]   (222.63,2930.75) ;
\draw [shift={(222.63,2930.75)}, rotate = 0] [color={rgb, 255:red, 208; green, 2; blue, 27 }  ,draw opacity=1 ][fill={rgb, 255:red, 208; green, 2; blue, 27 }  ,fill opacity=1 ][line width=0.75]      (0, 0) circle [x radius= 3.35, y radius= 3.35]   ;
\draw [color={rgb, 255:red, 208; green, 2; blue, 27 }  ,draw opacity=1 ]   (205.44,2930.07) ;
\draw [shift={(205.44,2930.07)}, rotate = 0] [color={rgb, 255:red, 208; green, 2; blue, 27 }  ,draw opacity=1 ][fill={rgb, 255:red, 208; green, 2; blue, 27 }  ,fill opacity=1 ][line width=0.75]      (0, 0) circle [x radius= 3.35, y radius= 3.35]   ;
\draw [color={rgb, 255:red, 208; green, 2; blue, 27 }  ,draw opacity=1 ]   (186.92,2930.07) ;
\draw [shift={(186.92,2930.07)}, rotate = 0] [color={rgb, 255:red, 208; green, 2; blue, 27 }  ,draw opacity=1 ][fill={rgb, 255:red, 208; green, 2; blue, 27 }  ,fill opacity=1 ][line width=0.75]      (0, 0) circle [x radius= 3.35, y radius= 3.35]   ;

\draw (144.58,2909.2) node [anchor=north west][inner sep=0.75pt]  [xslant=-0.04] [align=left] {\mbox{-}2M};
\draw (470.8,2909.89) node [anchor=north west][inner sep=0.75pt]  [xslant=-0.04] [align=left] {$\displaystyle 2M$};
\draw (332.51,3019.87) node [anchor=north west][inner sep=0.75pt]   [align=left] {\mbox{-}M};
\draw (332.57,2815.51) node [anchor=north west][inner sep=0.75pt]   [align=left] {$\displaystyle M$};
\draw (473.91,2845.75) node [anchor=north west][inner sep=0.75pt]   [align=left] {$\displaystyle \gamma _{4}$};
\draw (422.8,3019.89) node [anchor=north west][inner sep=0.75pt]   [align=left] {$\displaystyle \gamma _{3}$};
\draw (146.7,2982.51) node [anchor=north west][inner sep=0.75pt]   [align=left] {$\displaystyle \gamma _{2}$};
\draw (216.43,2814.54) node [anchor=north west][inner sep=0.75pt]   [align=left] {$\displaystyle \gamma _{1}$};

\end{tikzpicture}

\centering \caption{The contour $\gamma$}
\label{sheklaku} 
\end{figure}
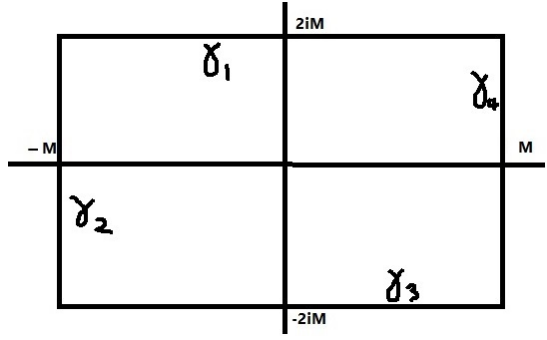


Now, one easily sees that the contribution of $\gamma_1,\gamma_3$ is negligible, and at least for $M$ large relative to $r$ the main contribution will come from the part of the contour near the real axis.  These two parts of the contour can be parametrized as $z = 2M + i t$ and $z = -2M + it$, $|t| \leq 2M$. So, for M large relative to $r$ we will have that the integral is 
\begin{align*} 
\sim \frac{1}{\pi} \int_{-2M}^{2M} \frac{(-4M it + t^2)^r}{\sin(\pi it)} dt &= \frac{1}{\pi}  \int_{-2M}^{2M} \frac{(-4M it)^r}{\sin(\pi it)} dt + \frac{1}{\pi} \int_{-2M}^{2M} \frac{r(-4Mit)^{r-1} t^2 } {\sin(\pi it)} dt + O(M^{r-2}) \\
&\sim 0 + \frac{r}{\pi} (-4Mi)^{r-1} \int_{-\infty}^{\infty}  \frac{t^{r+1}}{\sin(\pi it)} dt \sim c M^{r-1},
\end{align*}
for a constant $c$ that depends only on $r$. Note that the first term of RHS is zero by symmetry. This means that $f_r(M)$ is of degree $r-1$. 
Also we bound the size of $f_r(M)$ from above in the range $r \ll \frac{M}{\log(M)}$.
$$ \int_{\gamma} f(z) dz \sim c M^{r-1} \ll e^{r(\log(r)+\log\log(r))}M^{r-1}.$$
\qed


\end{document}